\documentclass[11pt]{article}
\usepackage[a4paper,margin=26mm]{geometry}
\usepackage{amsmath,amssymb,amsthm,mathtools,mathrsfs,array,booktabs}
\usepackage[hidelinks]{hyperref}
\usepackage[all]{xy}

\newtheorem{theorem}{Theorem}[section]
\newtheorem{proposition}[theorem]{Proposition}
\newtheorem{lemma}[theorem]{Lemma}
\newtheorem{corollary}[theorem]{Corollary}
\theoremstyle{definition}

\newtheorem{definition}[theorem]{Definition}
\theoremstyle{remark}
\newtheorem{remark}[theorem]{Remark}
\newcommand{\fa}{\mathfrak a}
\newcommand{\fb}{\mathfrak b}
\newcommand{\Hp}{H_{+}}
\newcommand{\Fp}{F_{+}}
\newcommand{\Hm}{H_{-}}
\newcommand{\Fm}{F_{-}}
\newcommand{\Rplus}{\mathcal R_{+}}
\newcommand{\Rminus}{\mathcal R_{-}}
\newcommand{\Dres}{\mathcal D_{\mathrm{res}}}
\newcommand{\Dbr}{\Delta_{\mathrm{br}}}
\newcommand{\thwone}{\theta_{w_1}}
\newcommand{\thwtwo}{\theta_{w_2}}
\newcommand{\PP}{\mathbb P}
\newcommand{\QQ}{\mathbb Q}
\newcommand{\CC}{\mathbb C}
\newcommand{\OO}{\mathcal O}
\newcommand{\Res}{\operatorname{Res}}
\newcommand{\Li}{\operatorname{Li}}
\title{Mirror symmetry for gCICY threefolds (I):\\ a $[3,-1]$~block and a non-Gorenstein toric phase}
\author{Atsushi Kanazawa}
\date{}

\begin{document}
\maketitle
\begin{abstract}
We construct and analyze a mirror family for the generalized complete intersection Calabi--Yau threefold
\[
X_+=\left[\begin{array}{c|cc}\PP^1&3&-1\\ \PP^4&2&3\end{array}\right], \qquad (h^{1,1},h^{2,1})=(2,46).
\]
An auxiliary-variable presentation turns the negative-degree section into a variation of toric GIT.
Crossing one wall flops 16 disjoint $(-1,-1)$-curves and yields a Calabi--Yau threefold $X_-$, a nef complete intersection in a toric $\QQ$-Fano eightfold, which carries a genus-one fibration of index 4.
Because this ambient space is not Gorenstein, the Batyrev--Borisov construction does not apply. 
We use the Hori--Vafa equations only to select a 2-parameter Laurent family, and construct a smooth projective crepant compactification $Y$ of the Laurent model with $(h^{1,1},h^{2,1})=(46,2)$; for very general parameters the Mordell--Weil group is $\mathbb Z/4\mathbb Z$.
The periods of $Y$ satisfy a rank-6 Picard--Fuchs system with two maximally unipotent boundary points, whose indicial algebras are isomorphic to the rational even cohomology rings of $X_-$ and $X_+$.
The resulting genus-zero predictions agree with independent counts of $176$ vertical lines and $100$ vertical conics on the genus-one fibration $X_-\to\PP^2$.
\end{abstract}

\tableofcontents

\section{Introduction}\label{sec:intro}
Complete intersection Calabi--Yau manifolds in products of projective spaces are among the most explicit compact Calabi--Yau manifolds: their defining equations, intersection theory, and many topological invariants are encoded by a configuration matrix.
Generalized complete intersection Calabi--Yau manifolds (gCICYs), introduced by Anderson, Apruzzi, Gao, Gray and Lee \cite{AAGGL}, enlarge this construction by allowing negative entries in later columns of the matrix.
The point is not that one writes a polynomial of negative degree on the ambient space.
Rather, one first cuts out an intermediate variety, on which a line bundle with no ambient sections can acquire sections after restriction.
Such a section, which in general does not extend to the ambient space, is called a \emph{generalized section} \cite{AAGGL}.
Garbagnati and van Geemen made this mechanism explicit by
\v{C}ech cohomology and treated the example studied in the present paper
\cite{GG}.  Subsequent work developed the fibration and sheaf-cohomological
aspects of gCICYs \cite{BHfib,JiaLin}.
In a related direction, Berglund and H\"ubsch proposed mirror constructions for non-Fano toric ambient spaces, especially Hirzebruch-type families, based on Laurent polynomials with non-reflexive and non-convex Newton data \cite{BHmirror,BHTyurin}; their starting point is a non-Fano toric ambient space, whereas ours is a generalized section and its birational chambers.

These constructions raise a natural question: does a gCICY with a negative entry have a mirror partner? The standard toric mirror constructions do not answer it directly.
The Batyrev--Borisov construction starts from a nef partition---a decomposition of the toric boundary into effective nef Cartier divisors summing to the anticanonical class---in a Gorenstein toric Fano variety \cite{BB}, whereas the defining line bundle of a gCICY is not nef on its ambient space.
The aim of this paper is to answer the question in one example,
\begin{equation}\label{eq:Xplusintro}
X_+=\left[\begin{array}{c|cc}\PP^1&3&-1\\ \PP^4&2&3\end{array}\right], \qquad (h^{1,1},h^{2,1})=(2,46),
\end{equation}
and to show that its birational, Hodge-theoretic, and genus-zero aspects fit into a single explicit picture.

On the A-side, an auxiliary-variable presentation realizes the generalized section by a system of polynomial equations for an action of the rank-2 torus $(\CC^*)^2$.
Its two chambers give $X_+$ and a second Calabi--Yau threefold $X_-$, related by a flop of 16 disjoint $(-1,-1)$-curves, and $X_-$ carries a genus-one fibration of index 4.
The two chambers are asymmetric.
In the chamber of $X_+$ the toric ambient space is not $\QQ$-Fano and the second defining class, corresponding to $\OO_{\PP^1\times\PP^4}(-1,3)$, is not nef.
In the chamber of $X_-$ the ambient space is a toric $\QQ$-Fano eightfold, all 5 defining classes are nef, and the toric boundary admits a natural decomposition into 5 parts, one for each defining class.

This resembles the input of Batyrev--Borisov mirror symmetry, but only up to a precise point: the ambient space of $X_-$ has a stratum of $\frac12(1,1,1,1,1)$ quotient singularities, so it is not Gorenstein, and one of the 5 defining classes is only $\QQ$-Cartier.
Hence the Batyrev--Borisov construction does not apply.
We instead use the corresponding Hori--Vafa torus equations to produce a
2-parameter Laurent family and construct a smooth projective crepant
compactification $Y$ of its general member, through a double-cover model and a
Weierstrass model. Its geometry and periods are then studied directly
(Sections~\ref{sec:Bmodel} and~\ref{sec:periods}).
The general mechanism is developed in \cite{KanII} for gCICYs with a \emph{$[3,-1]$~block}, $[\PP^1|3\ {-1};\ B|L\ \,L{+}R]$, where $B$ is a smooth projective fourfold and $L,R\in\operatorname{Pic}(B)$ satisfy $2L+R=-K_B$.
For $B=\PP^4$, $L=\OO(a)$ and $R=\OO(r)$, the nef chamber is Gorenstein, and the Batyrev--Borisov construction applies, exactly when $a\mid r$; $X_+$ is the case $a=2$, $r=1$.

The basic geometric picture is
\[
\xymatrix{
X_+   \ar@{-->}[rr]^{\text{flop}}  \ar@{<->}[rd]_{\text{mirror}}  &  &  X_-  \ar@{<->}[ld]^{\text{mirror}}  \\
 &  Y &
}
\]
where $Y$ is mirror symmetric to $X_-$ and to $X_+$ at its two maximally unipotent (large-complex-structure) boundary points, in the sense of Theorem~\ref{thm:main}(3).
The route from $X_+$ to $Y$ passes through auxiliary variables and variation of GIT (Section~\ref{sec:Amodels}), 
the Hori--Vafa equations, a double cover, a Weierstrass model and an explicit crepant resolution (Section~\ref{sec:Bmodel}), 
the periods and their Picard--Fuchs system (Section~\ref{sec:periods}), and finally the Yukawa couplings and the prepotential, 
whose genus-zero predictions are tested by an independent enumerative calculation on $X_-$ (Section~\ref{sec:quantum}).
The main results are summarized below.

\begin{theorem}[Main results]\label{thm:main}
For a general threefold $X_+$ in \eqref{eq:Xplusintro} the following hold.
\begin{enumerate}
\item There is a Calabi--Yau threefold $X_-$ obtained from $X_+$ by
flopping 16 disjoint $(-1,-1)$-curves.
The threefold $X_-$ admits a genus-one fibration over $\PP^2$ of index exactly 4.
In natural integral divisor bases $(\Hp,\Fp)$ on $X_+$ and $(\Hm,\Fm)$ on $X_-$, with $\Hm=\Hp$ and $\Fm=2\Hp-\Fp$, the cubic data are
\[
\Hp^3=7,\quad \Hp^2\Fp=6,\quad \Fp^2=0,
\]
\[
\Hm^3=7,\quad \Hm^2\Fm=8,\quad \Hm\Fm^2=4,\quad \Fm^3=0.
\]
(Propositions~\ref{prop:flop}, \ref{prop:topdata} and~\ref{prop:index4}.)
\item For $(z_1,z_2)$ in a dense Zariski-open subset of
$(\CC^*)^2$, the Laurent model \eqref{eq:laurent} admits a smooth projective crepant model $Y=Y_{z_1,z_2}$ with
\[
e(Y)=88,\qquad (h^{1,1}(Y),h^{2,1}(Y))=(46,2).
\]
For very general parameters the Mordell--Weil group of its elliptic presentation is $\mathbb Z/4\mathbb Z$.
(Propositions~\ref{prop:MW}, \ref{prop:rho46} and~\ref{prop:euler88}.)
\item The fundamental period is annihilated by the Picard--Fuchs system
\eqref{eq:L2}--\eqref{eq:L3}, which has generic rank 6 and two maximally
unipotent boundary points.
Their indicial algebras are isomorphic to $\Rminus\otimes\CC$ and $\Rplus\otimes\CC$, respectively, where
\[
\Rminus=\QQ[\Hm,\Fm]/(4\Hm^2-8\Hm\Fm+9\Fm^2,\Fm^3)
\]
and
\[
\Rplus=\QQ[\Hp,\Fp]/(\Fp^2,\Hp^2(6\Hp-7\Fp))
\]
are presentations of the rational even cohomology rings of $X_-$ and $X_+$, with divisor classes placed in degree 1.
(Propositions~\ref{prop:PF}, \ref{prop:minuschart}, \ref{prop:pluschart}, Corollary~\ref{cor:Aindminus}.)
\item The system admits a nondegenerate rational horizontal alternating
form.
Its 4 Yukawa functions are the rational functions in
\eqref{eq:Yukawas}.  The corresponding mirror maps and formal prepotential give genus-zero
predictions.
Their first vertical genus-zero predictions agree with independent intersection-theoretic calculations of $176$ vertical lines and $100$ vertical conics on $X_-$; here ``vertical'' means contained in a fiber of the genus-one fibration $X_-\to\PP^2$.
(Propositions~\ref{prop:yukawa}, \ref{prop:firstrow}, \ref{prop:176lines} and~\ref{prop:100conics}.)
\end{enumerate}
\end{theorem}

The index 4 of the genus-one fibration on $X_-$ matches the order of the torsion Mordell--Weil group $\mathbb Z/4\mathbb Z$ on the B-side, as expected from \cite{ORS}.
The two birational chambers in the movable cone correspond to two nilpotent boundary cones generated by logarithms of the local unipotent monodromies, in accordance with \cite{HT}.

\paragraph{Convention.}
We work over $\CC$.
A Calabi--Yau threefold $X$ is a smooth projective threefold with $K_X\sim0$ and $H^1(X,\OO_X)=H^2(X,\OO_X)=0$.
``General'' means on a dense Zariski-open set and ``very general'' means outside a countable union of proper closed subsets.
We use \emph{chamber} both for the two GIT chambers of the auxiliary toric problem and for the corresponding A-model birational models $X_\pm$, and
\emph{chart} for the coordinates near the two B-model boundary points.

\paragraph{Acknowledgments.}
The author thanks Makoto Miura for valuable discussions.
Part of this work was inspired by interactions during the author's visit to National Taiwan University in spring 2026, whose hospitality is gratefully acknowledged.
This work is supported by JSPS Grants-in-Aid Kiban (C) 22K03296 and Kiban (A) 23H00083.

\section{gCICYs and A-model chambers}
\label{sec:Amodels}

\subsection{Generalized complete intersections}

A gCICY in a product $A$ of projective spaces is an \emph{iterated} zero locus: starting with $M_0=A$, one chooses $s_j\in H^0(M_{j-1},\OO_A(q_j)|_{M_{j-1}})$ and puts $M_j=Z_{M_{j-1}}(s_j)$
\cite{AAGGL}.  The entries of a later column $q_j$ of the configuration
matrix may be negative, provided the restricted line bundle has a nonzero section; $s_j$ is then a generalized section.
Such sections were made explicit by \v{C}ech cohomology in \cite{GG}.
Repeated adjunction shows that the usual row-sum condition still makes the canonical bundle trivial, while smoothness and the vanishing of $h^i(\OO)$ for $0<i<\dim$ must be checked separately.

We take $A=\PP^1\times\PP^4$ with homogeneous coordinates $[u:v]$ and $x=[x_0:\cdots:x_4]$, and we write $F$ and $H$ for the pullbacks of the hyperplane classes of $\PP^1$ and $\PP^4$.
For the configuration
\eqref{eq:Xplusintro} we have $q_1=(3,2)$ and $q_2=(-1,3)$.  Let
\[
M=Z_A(P),\qquad P\in H^0(A,\OO_A(3,2)),
\]
be a smooth hypersurface; this holds for general $P$ because $\OO_A(3,2)$ is very ample.
By adjunction $K_M=\OO_M(-2,-5)\otimes\OO_M(3,2)=\OO_M(1,-3)$, so the second line bundle of the configuration is
\[
\OO_M(-1,3)=\omega_M^{-1}.
\]

\begin{lemma}\label{lem:sectionsM}
\begin{enumerate}
\item $h^i(M,\OO_M)=1$ for $i=0$ and $h^i(M,\OO_M)=0$ for $i>0$.
\item Although $H^0(A,\OO_A(-1,3))=0$, the restriction has sections:
\begin{equation}\label{eq:gcicycohomology}
H^0(M,\OO_M(-1,3))
 \simeq H^1(A,\OO_A(-4,1))
 \simeq H^1(\PP^1,\OO(-4))\otimes H^0(\PP^4,\OO(1)),
\end{equation}
a space of dimension $3\cdot5=15$.
\item If the zero locus $X_+$ of a section $s\in H^0(M,\omega_M^{-1})$ is
smooth of dimension 3, then $X_+$ is a Calabi--Yau threefold.
\end{enumerate}
\end{lemma}

\begin{proof}
(1) In the sequence $0\to\OO_A(-3,-2)\to\OO_A\to\OO_M\to0$ the first term is acyclic by the K\"unneth formula, since $H^*(\PP^4,\OO(-2))=0$.
(2) Twist the same sequence by $\OO_A(-1,3)$:
\[
0\longrightarrow\OO_A(-4,1)
 \xrightarrow{\times P}\OO_A(-1,3)
 \longrightarrow\OO_M(-1,3)\longrightarrow0 .
\]
By the K\"unneth formula, the groups $H^0$ and $H^1$ of $\OO_A(-1,3)$ vanish because $\OO_{\PP^1}(-1)$ is acyclic.
The connecting map therefore gives
\eqref{eq:gcicycohomology}.

(3) Adjunction gives $K_{X_+}=(K_M+c_1(\omega_M^{-1}))|_{X_+}=0$.
The sequence $0\to\omega_M\to\OO_M\to\OO_{X_+}\to0$, part (1), and Serre duality $h^i(\omega_M)=h^{4-i}(\OO_M)$ give $h^0(\OO_{X_+})=1$, so $X_+$ is connected, and $h^1(\OO_{X_+})=h^2(\OO_{X_+})=0$.
\end{proof}

Thus the generalized section of the configuration is precisely an anticanonical section of the intermediate fourfold $M$, and the negative entry records a restricted section, not a polynomial of negative degree on $A$.
The anticanonical system $|\omega_M^{-1}|$ is not base-point free; we determine its base locus in Section~\ref{sec:aux} and prove in Corollary~\ref{cor:Xplussmooth} that its general member
\[
X_+=Z_M(s) =\left[\begin{array}{c|cc}\PP^1&3&-1\\ \PP^4&2&3\end{array}\right]
\]
is smooth.
It is then the Calabi--Yau threefold with Hodge numbers $(2,46)$ studied in \cite{AAGGL,GG}.

The projection $X_+\to\PP^1$ is a K3 fibration.
Over a point $[u:v]$ the fiber is cut out in $\PP^4$ by the quadric $P(u,v,\cdot)$ and by the restriction of $s$, which by \eqref{eq:gcicycohomology} is a cubic form on the fiber; the general fiber is therefore the degree-6 K3 surface $[\PP^4\,|\,2\ 3]$.

\subsection{Auxiliary presentation and two chambers}
\label{sec:aux}

\paragraph{The syzygy of the cubic.}
Write the first equation as
\[
P=f_0(x)u^3+f_1(x)u^2v+f_2(x)uv^2+f_3(x)v^3,
\]
where the $f_i$ are quadratic forms, and put $\mathbf f=(f_0,f_1,f_2,f_3)^t\in H^0(A,\OO_A(0,2)^{\oplus4})$.
On $\PP^1$ consider
\begin{equation}\label{eq:syzygyP1}
0\longrightarrow\OO(-1)^{\oplus3}
\xrightarrow{D(u,v)}\OO^{\oplus4}
\xrightarrow{(u^3,u^2v,uv^2,v^3)}\OO(3)\longrightarrow0, \qquad
D(u,v)=\begin{pmatrix}
-v&0&0\\ u&-v&0\\ 0&u&-v\\ 0&0&u
\end{pmatrix}.
\end{equation}
It is exact.
The evaluation map is surjective because $u^3$ and $v^3$ have no common zero, and
\[
(u^3,u^2v,uv^2,v^3)\,D(u,v)=0 .
\]
Since $D(u,v)$ has rank 3 at every point, its image is a subbundle of rank 3 contained in, hence equal to, the kernel.
Pulling back to $A$ and tensoring by $\OO_A(0,2)$ gives the exact sequence of vector bundles
\begin{equation}\label{eq:syzygyA}
0\longrightarrow\OO_A(-1,2)^{\oplus3}
\xrightarrow{D(u,v)}\OO_A(0,2)^{\oplus4}
\xrightarrow{(u^3,u^2v,uv^2,v^3)}
\OO_A(3,2)\longrightarrow0,
\end{equation}
under which $\mathbf f\mapsto P$.
Since \eqref{eq:syzygyA} is locally split, it stays exact after restriction to $M$, and $\mathbf f|_M$ is a section of the kernel $\OO_M(-1,2)^{\oplus3}$.
Hence there is a unique $y=(y_1,y_2,y_3)^t\in H^0\!\left(M,\OO_M(-1,2)^{\oplus3}\right)$ with
\begin{equation}\label{eq:Dylift}
D(u,v)y=\mathbf f .
\end{equation}
Conversely, multiplying $D(u,v)y=\mathbf f$ by $(u^3,u^2v,uv^2,v^3)$ recovers $P=0$.
Thus the $y_j$ are not independent new geometric data on $M$: they are uniquely determined as the coordinates of the syzygy expressing $P=0$.
On the chart $v\ne0$,
\[
y_1=-\frac{f_0}{v},\qquad y_2=-\frac{uf_0+vf_1}{v^2},\qquad y_3=-\frac{u^2f_0+uvf_1+v^2f_2}{v^3},
\]
and the fourth row of $D(u,v)y=\mathbf f$ is equivalent to $P=0$; on the chart $u\ne0$ one has analogous formulas, and at $v=0$ one reads off $y_j=f_j$ $(j=1,2,3)$ on the chart $u=1$.
Since $D(u,v)$ is injective on every fiber,
\begin{equation}\label{eq:yzero}
\{y=0\}=\{\mathbf f=0\}\cap M=\PP^1\times Z_{\PP^4}(f_0,\ldots,f_3).
\end{equation}

\begin{lemma}\label{lem:ybasis}
For general $P$ the following hold.
\begin{enumerate}
\item $y_1,y_2,y_3$ form a basis of
$H^0(M,\OO_M(-1,2))\simeq H^1(\PP^1,\OO(-4))$.
\item Multiplication induces an isomorphism
\[
H^0(M,\OO_M(-1,2))\otimes H^0(\PP^4,\OO(1))
\xrightarrow{\ \sim\ }H^0(M,\OO_M(-1,3)).
\]
Consequently every generalized section has the form
\[
s=\ell_1y_1+\ell_2y_2+\ell_3y_3,\qquad \ell_j\in H^0(\PP^4,\OO(1)).
\]
\end{enumerate}
\end{lemma}

\begin{proof}
The argument of Lemma~\ref{lem:sectionsM}(2), applied to the sequence
\[
0\to\OO_A(-4,0)\to\OO_A(-1,2)\to\OO_M(-1,2)\to0,
\]
gives an isomorphism
\[
\delta:H^0(M,\OO_M(-1,2))\xrightarrow{\ \sim\ }H^1(A,\OO_A(-4,0))
\simeq H^1(\PP^1,\OO(-4)),
\]
onto a space of dimension 3.
If $\sum_{j=1}^3c_jy_j=0$, restriction to $M\cap\{v=0\}=\{f_0=0\}\subset\PP^4$ gives $\sum_{j=1}^3c_jf_j\in\CC f_0$, which for general quadrics forces $c=0$.
This proves (1).
The connecting maps $\delta$ for $\OO_A(-1,2)$ and $\OO_A(-1,3)$ commute with multiplication by $\ell\in H^0(\PP^4,\OO(1))$.
On the targets this multiplication is, by K\"unneth, the isomorphism
\[
\bigl(H^1(\PP^1,\OO(-4))\otimes H^0(\PP^4,\OO)\bigr)\otimes H^0(\PP^4,\OO(1))
\xrightarrow{\ \sim\ }H^1(\PP^1,\OO(-4))\otimes H^0(\PP^4,\OO(1)).
\]
This proves (2); the dimension count $3\cdot5=15$ agrees with
\eqref{eq:gcicycohomology}.
\end{proof}

Combining the lemma with the previous discussion, the gCICY is represented by the polynomial system
\begin{equation}\label{eq:auxsystem}
D(u,v)y=(f_0,f_1,f_2,f_3)^t, \qquad
\ell_1y_1+\ell_2y_2+\ell_3y_3=0,
\end{equation}
in the variables $u,v,x_0,\ldots,x_4,y_1,y_2,y_3$ of multidegrees
\[
\deg(u)=\deg(v)=(1,0),\quad
\deg(x_i)=(0,1),\quad
\deg(y_j)=(-1,2).
\]
The negative degree has not disappeared; it has moved into the multidegrees of the auxiliary variables.
This is what makes it possible to interpret the same homogeneous system in two different GIT chambers.

\paragraph{Toric GIT and the wall between the chambers.}
Let
\[
W=\CC^{10}=\operatorname{Spec}\CC[u,v,x_0,\ldots,x_4,y_1,y_2,y_3]
\]
and let $\mathbb G=(\CC^*)^2$ act with the weights
\[
q_U=(1,0)\ \text{on }u,v,\qquad q_X=(0,1)\ \text{on }x_0,\ldots,x_4,\qquad q_Y=(-1,2)\ \text{on }y_1,y_2,y_3 .
\]
The equations \eqref{eq:auxsystem} are $\mathbb G$-homogeneous and define a $\mathbb G$-invariant affine subscheme $Z\subset W$.
We regard $Z$, together with the $\mathbb G$-action, as the VGIT master space for the two Calabi--Yau phases; $W$ is its ambient affine Cox space.
A character $\vartheta\in\operatorname{Hom}(\mathbb G,\CC^*)\otimes\mathbb R\simeq\mathbb R^2$ chooses a GIT linearization, 
and a point of $W$ is $\vartheta$-semistable precisely when $\vartheta$ lies in the cone generated by the weights of its nonzero coordinates.
The effective cone is $\operatorname{Cone}(q_U,q_Y)$, and the ray through $q_X=\frac12(q_U+q_Y)$ divides it into the two chambers
\[
\mathfrak C_+=\operatorname{Cone}(q_U,q_X),\qquad
\mathfrak C_-=\operatorname{Cone}(q_X,q_Y),
\]
separated by the wall $\mathbb R_{\ge0}q_X$.
For $\vartheta$ in the interior of a chamber, semistability equals stability and is described as follows:
\begin{align*}
\vartheta\in\mathfrak C_+:&\quad (u,v)\ne0\ \text{and}\ (x,y)\ne0,\\
\vartheta\in\mathfrak C_-:&\quad y\ne0\ \text{and}\ (u,v,x)\ne0 .
\end{align*}
Within a chamber the quotient is constant, and crossing the wall changes it birationally; this is the elementary toric instance of variation of GIT.
In Cox's language \cite{Cox} the 10 variables are homogeneous coordinates and the two rows of their weights form the class-group grading.

In both chambers the stabilizers are trivial except along $\{x=0\}$.
At a semistable point with $x=0$ one has $(u,v)\ne0$ and $y\ne0$, and an element $(\lambda_1,\lambda_2)\in \mathbb G$ fixes it if and only if $\lambda_1=1$ and $\lambda_1^{-1}\lambda_2^2=1$.
The stabilizer is therefore $\mu_2$, acting by $-1$ on the 5 transverse coordinates $x_0,\ldots,x_4$.
This stratum plays a role in Section~\ref{sec:nefpartition}; the threefolds $X_\pm$ avoid it.

\begin{proposition}\label{prop:pluschamber}
Let $\vartheta\in\mathfrak C_+$.
Every $\vartheta$-semistable point of $Z$ has $(u,v)\ne0$ and $x\ne0$, $\mathbb G$ acts freely on $Z^{ss}$, and
\[
Z^{ss}/\mathbb G\ \simeq\ X_+ .
\]
\end{proposition}

\begin{proof}
If $x=0$ on $Z$, then $\mathbf f(x)=0$, so $D(u,v)y=0$ and $y=0$ because $(u,v)\ne0$ and $D(u,v)$ is injective; then $(x,y)=0$, which is unstable.
On the open set $\{(u,v)\ne0,\ x\ne0\}$ the action is free and the quotient is the total space of $\OO_A(-1,2)^{\oplus3}$ over $A$, with fiber coordinates $y$.
In it, the first 4 equations define the graph of the section $y$ of \eqref{eq:Dylift} over $M$, by the uniqueness of the lift, and the last equation cuts out $Z_M(s)=X_+$ with $s=\ell\cdot y$.
\end{proof}

\paragraph{The minus chamber: eliminating $(u,v)$.}
In the chamber $\mathfrak C_-$ the coordinates $y$ are not all zero and define a point of $\PP^2_y$.
Rewrite the first 4 equations as
\[
D(u,v)y=\mathsf A(y)\binom uv=\mathbf f(x), \qquad
\mathsf A(y)=\begin{pmatrix}
0&-y_1\\ y_1&-y_2\\ y_2&-y_3\\ y_3&0
\end{pmatrix}.
\]
For every $[y]\in\PP^2$ the matrix $\mathsf A(y)$ has rank 2, so it defines an injective map of vector bundles and a rank-2 quotient bundle $\mathcal Q$:
\begin{equation}\label{eq:Qbundle}
0\longrightarrow\OO_{\PP^2}(-1)^{\oplus2}
\xrightarrow{\mathsf A(y)}\OO_{\PP^2}^{\oplus4}
\xrightarrow{\pi_{\mathcal Q}}\mathcal Q\longrightarrow0, \qquad
\mathcal Q_y=\CC^4/\operatorname{Im}\mathsf A(y).
\end{equation}
For fixed $(y,x)$, the equation $\mathsf A(y)(u,v)^t=\mathbf f(x)$ has a solution if and only if the class of $\mathbf f(x)$ in $\mathcal Q_y$ vanishes, and the solution is then unique.
Thus eliminating $(u,v)$ from the first 4 equations is exactly the condition
\begin{equation}\label{eq:fbar}
\bar{\mathbf f}:=\pi_{\mathcal Q}(\mathbf f)=0, \qquad
\bar{\mathbf f}\in H^0\!\left(\PP^2\times\PP^4, \mathcal Q\boxtimes\OO_{\PP^4}(2)\right),
\end{equation}
the twist $\OO_{\PP^4}(2)$ appearing because each $f_i$ is a quadric.
The last equation $\ell_1(x)y_1+\ell_2(x)y_2+\ell_3(x)y_3=0$ is bilinear in $y$ and $x$, i.e.\ it is the vanishing of a section $\bar\ell\in H^0(\PP^2\times\PP^4,\OO(1,1))$.
Put
\begin{equation}\label{eq:Xminus}
X_-=Z(\sigma_- )\subset\PP^2_y\times\PP^4_x, \qquad
\sigma_-=(\bar{\mathbf f},\bar\ell)
\in H^0\!\left(\mathcal E\right),\qquad
\mathcal E=(\mathcal Q\boxtimes\OO_{\PP^4}(2))\oplus\OO(1,1).
\end{equation}
The first summand is the obstruction to solving the 4 equations for $(u,v)$, and the second is the original generalized section.

\begin{proposition}\label{prop:minuschamber}
\begin{enumerate}
\item Let $\vartheta\in\mathfrak C_-$.  Every $\vartheta$-semistable point
of $Z$ has $x\ne0$ and $y\ne0$, $\mathbb G$ acts freely on $Z^{ss}$, and $Z^{ss}/\mathbb G\simeq X_-$.
\item $H^0(\mathcal E)=H^0(\PP^4,\OO(2))^{\oplus4}\oplus
(\CC^3\otimes H^0(\PP^4,\OO(1)))$; that is, every section of $\mathcal E$ is $\sigma_-$ for some data $(\mathbf f,\ell)$.
\item $\mathcal E$ is globally generated of rank 3 and
$\det\mathcal E=\OO(3,5)=-K_{\PP^2\times\PP^4}$.
For general data $X_-$ is a Calabi--Yau threefold.
\end{enumerate}
\end{proposition}

\begin{proof}
(1) If $x=0$ on $Z$, then $\mathsf A(y)(u,v)^t=0$ forces $(u,v)=0$, so $(u,v,x)=0$, which is unstable.
On $\{x\ne0,\ y\ne0\}$ an element of $\mathbb G$ fixing a point has $\lambda_2=1$ (from $x$) and $\lambda_1^{-1}\lambda_2^2=1$ (from $y$), so the action is free.
The quotient of this open set is the total space of $(\OO_{\PP^2}(-1)\boxtimes\OO_{\PP^4}(2))^{\oplus2}$ over $\PP^2\times\PP^4$, with fiber coordinates $(u,v)$, because $\deg u=\deg v=(1,0)=-q_Y+2q_X$.
In it the 4 equations $\mathsf A(y)(u,v)^t=\mathbf f(x)$ define the graph of the unique solution $(u,v)$ over $Z(\bar{\mathbf f})$, since $\mathsf A$ is an injective bundle map; the last equation gives $Z(\sigma_-)$.

(2) Tensor \eqref{eq:Qbundle} by $\OO_{\PP^4}(2)$; since $\OO_{\PP^2}(-1)$ is acyclic, $H^0(\mathcal Q\boxtimes\OO(2))= H^0(\OO^{\oplus4}\boxtimes\OO(2))$.
The statement for $\OO(1,1)$ is K\"unneth.

(3) $\mathcal Q$ is a quotient of a trivial bundle, hence globally generated, and $\det\mathcal Q=\OO_{\PP^2}(2)$ by \eqref{eq:Qbundle}.
By (2) and Bertini's theorem for globally generated bundles, the zero locus of a general section, i.e.\ of general data, is smooth of codimension 3, and adjunction gives $K_{X_-}\sim0$.
For the cohomology, the Koszul complex of $\sigma_-$ has terms
\begin{gather*}
\mathcal E^*=(\mathcal Q^*\boxtimes\OO(-2))\oplus\OO(-1,-1),\qquad
\Lambda^3\mathcal E^*=\OO(-3,-5),\\
\Lambda^2\mathcal E^*=\OO(-2,-4)\oplus
(\mathcal Q^*(-1)\boxtimes\OO(-3)).
\end{gather*}
The terms $\mathcal E^*$ and $\Lambda^2\mathcal E^*$ are acyclic by K\"unneth, since $\OO_{\PP^4}(-2)$, $\OO_{\PP^2}(-1)$, $\OO_{\PP^4}(-4)$ and $\OO_{\PP^4}(-3)$ are acyclic, and $\Lambda^3\mathcal E^*$ is the canonical bundle.
Hence $h^i(\OO_{X_-})=1,0,0,1$ for $i=0,1,2,3$; in particular $X_-$ is connected.
\end{proof}

\paragraph{Divisor classes on the two chambers.}
On $X_+$ let $\Hp$ and $\Fp$ be the restrictions of $H$ and $F$; thus $\Fp$ is the class of a K3 fiber.
On $X_-$ let $\Hm$ be the pullback of the hyperplane class of $\PP^4_x$ and $\Fm$ that of $\PP^2_y$.
The line bundles associated with the characters $q_X$ and $q_U$ restrict to $\Hp$ and $\Fp$ on $X_+$, and those associated with $q_X$ and $q_Y$ restrict to $\Hm$ and $\Fm$ on $X_-$.
Both chambers contain the open set of $Z$ where $x\ne0$, $y\ne0$ and $(u,v)\ne0$, and on its quotient the relation $q_Y=2q_X-q_U$ gives
\[
\Hm=\Hp,\qquad \Fm=2\Hp-\Fp .
\]
This identification extends to $\operatorname{Pic}(X_\pm)$ because the complement of the open set in $X_\pm$ has codimension at least 2, as the next proposition shows.

\paragraph{The 16-curve flop.}
\begin{proposition}\label{prop:flop}
For general data, $f_0,\ldots,f_3$ meet in 16 reduced points $x^{(1)},\ldots,x^{(16)}\in\PP^4$ at which $\ell=(\ell_1,\ell_2,\ell_3)$ does not vanish.
The base locus of the generalized sections on $M$ is the disjoint union of the curves $C_i=\PP^1\times\{x^{(i)}\}$, and $N_{C_i/X_+}\simeq\OO(-1)^{\oplus2}$.
The birational map $X_+\dashrightarrow X_-$ induced by \eqref{eq:auxsystem} is an isomorphism away from $\bigcup C_i$ and the lines
\[
C_i'=\{\ell_1(x^{(i)})y_1+\ell_2(x^{(i)})y_2+\ell_3(x^{(i)})y_3=0\}\times\{x^{(i)}\}
\subset\PP^2_y\times\PP^4_x,
\]
and it is the simultaneous flop of the 16 disjoint $(-1,-1)$-curves $C_i$.
\end{proposition}

\begin{proof}
The 4 general quadrics $f_j$ meet in $2^4=16$ reduced points, and general linear forms $\ell_j$ do not vanish simultaneously at them.
By Lemma~\ref{lem:ybasis} the generalized sections are the $\ell\cdot y$, and at a point of $M$ with $y_j\ne0$ one can choose $\ell_j$ nonzero there and the other $\ell_k$ zero; hence the base locus is $\{y=0\}$, which is $\bigcup C_i$ by
\eqref{eq:yzero}.

\emph{Normal bundles.}  We have $N_{C_i/A}=T_{x^{(i)}}\PP^4\otimes\OO_{C_i}
=\OO^{\oplus4}$, and $dP|_{C_i}=\sum_{j=0}^3u^{3-j}v^j\,df_j(x^{(i)})$, where the 4 linear forms $df_j(x^{(i)})$ on $T_{x^{(i)}}\PP^4$ are independent because the 16 points are reduced. Hence $N_{C_i/M}$ is the kernel of $\OO^{\oplus4}\to\OO(3)$, $(\xi_j)\mapsto\sum_{j=0}^3u^{3-j}v^j\xi_j$, which is $\OO(-1)^{\oplus3}$ by \eqref{eq:syzygyP1}. The section $y$ vanishes along $C_i$, and its derivative $dy:N_{C_i/M}\to\OO_M(-1,2)^{\oplus3}|_{C_i}=
\OO(-1)^{\oplus3}$ satisfies $D(u,v)\circ dy=d\mathbf f$, which is the
inclusion of $N_{C_i/M}$ into $\OO^{\oplus4}$; hence $dy$ is injective, and an isomorphism since both bundles have degree $-3$.
The derivative of $s=\ell\cdot y$ along $C_i$ is therefore the constant map $\OO(-1)^{\oplus3}\to\OO(-1)$ given by $\ell(x^{(i)})\ne0$.
It is surjective, so $X_+$ is smooth along $C_i$ and $N_{C_i/X_+}$ is its kernel $\OO(-1)^{\oplus2}$.

\emph{Where the chambers agree.}  By Propositions~\ref{prop:pluschamber}
and~\ref{prop:minuschamber}, $X_\pm$ are the quotients of the semistable parts of $Z$ for the two chambers.
A point of $Z$ with $x\ne0$, $y\ne0$ and $(u,v)\ne0$ is semistable in both.
The points of $X_+$ with $y=0$ form $\bigcup C_i$; the points of $X_-$ with $(u,v)=0$ are those with $\mathbf f(x)=0$, i.e.\ $x=x^{(i)}$, together with $\ell(x^{(i)})\cdot y=0$, which form $\bigcup C_i'$.
Hence $X_+\setminus\bigcup C_i\simeq X_-\setminus\bigcup C_i'$, and both complements have codimension 2.

\emph{The flop.}  The classes $\Hp$ and $\Hm$ are base-point free, pulled
back from $\PP^4$, and big because $X_\pm\to\PP^4$ are generically finite, as we now check.
The curves on $X_+$ contracted by $X_+\to\PP^4$ lie in the fibers $\{[u:v]\mid P(u,v,x)=0\}\cap\{s=0\}$, which are finite unless $\mathbf f(x)=0$, in which case the fiber is $C_i$.
Likewise the fiber of $X_-\to\PP^4$ over $x$ is finite unless $\mathbf f(x)=0$, in which case it is $C_i'$.
Let $X_\pm\to\overline X$ be the Stein factorizations; they are the contractions of $\bigcup C_i$ and $\bigcup C_i'$ to the same normal variety $\overline X$, which has 16 ordinary double points because $N_{C_i/X_+}=\OO(-1)^{\oplus2}$.
The class $\Fp$ satisfies $\Fp\cdot C_i=1$ and is ample relative to $X_+\to\overline X$, whereas its transform $2\Hm-\Fm$ on $X_-$ satisfies $(2\Hm-\Fm)\cdot C_i'=-1$, because $\Hm\cdot C_i'=0$ and $\Fm\cdot C_i'=1$.
By the uniqueness of the flop \cite[\S6.1]{KM}, $X_-\to\overline X$ is the flop of $X_+\to\overline X$, i.e.\ the simultaneous Atiyah flop of the 16 disjoint curves $C_i$.
\end{proof}

\begin{corollary}\label{cor:Xplussmooth}
For general data $X_+$ is a Calabi--Yau threefold.
\end{corollary}

\begin{proof}
Away from the base locus $\bigcup C_i$ of $|\omega_M^{-1}|$ this is Bertini's theorem, and along it the proof of Proposition~\ref{prop:flop}.
The Calabi--Yau conditions are Lemma~\ref{lem:sectionsM}(3).
\end{proof}

\subsection{Intersection theory and the index-4 genus-one fibration}
\label{sec:intersection}

\begin{proposition}\label{prop:topdata}
\begin{enumerate}
\item On $X_+$,
\[
\Hp^3=7,\qquad \Hp^2\Fp=6,\qquad \Fp^2=0,\qquad
c_2(X_+)=4\Hp^2+3\Hp\Fp,
\]
\[
c_2(X_+)\cdot\Hp=46,\qquad c_2(X_+)\cdot\Fp=24,\qquad e(X_+)=-88 .
\]
\item On $X_-$,
\[
\Hm^3=7,\quad \Hm^2\Fm=8,\quad \Hm\Fm^2=4,\quad \Fm^3=0, \qquad c_2(X_-)\cdot\Hm=46,\quad c_2(X_-)\cdot\Fm=36,
\]
and $e(X_-)=-88$.
\item For every $D\in\operatorname{Pic}(X_\pm)$ and any flopping curve
$C_i$,
\[
D^3\big|_{X_-}=D^3\big|_{X_+}-16\,(D\cdot C_i)^3,\qquad c_2(X_-)\cdot D=c_2(X_+)\cdot D+32\,(D\cdot C_i).
\]
\end{enumerate}
\end{proposition}

\begin{proof}
(1) Since $X_+$ is the zero locus of a section of $\OO_M(-1,3)$ on the hypersurface $M$, integration over $X_+$ is integration over $A$ against $[X_+]=(3F+2H)(3H-F)$, and
\[
c(T_{X_+})=\frac{(1+F)^2(1+H)^5}{(1+3F+2H)(1+3H-F)}\Big|_{X_+}.
\]
With $F^2=0$ and $\int_AFH^4=1$ one finds $\Hp^3=9-2=7$, $\Hp^2\Fp=6$, $c_2(X_+)=4\Hp^2+3\Hp\Fp$, and hence $c_2\cdot\Hp=4\cdot7+3\cdot6=46$, $c_2\cdot\Fp=4\cdot6=24$; the degree-3 part of $c(T_{X_+})$ gives $e(X_+)=-88$.
Since $h^{1,1}(X_+)=2$ \cite{AAGGL,GG}, this gives $h^{2,1}(X_+)=46$.

(2) On $\PP^2\times\PP^4$ one has $c(\mathcal Q)=(1-\Fm)^{-2} =1+2\Fm+3\Fm^2$, so
\[
[X_-]=c_2(\mathcal Q\boxtimes\OO(2))\,(\Fm+\Hm) =(3\Fm^2+4\Fm\Hm+4\Hm^2)(\Fm+\Hm) =7\Fm^2\Hm+8\Fm\Hm^2+4\Hm^3,
\]
and the normal bundle sequence gives $c_2(X_-)=2\Hm^2+5\Hm\Fm-2\Fm^2$.
Integration gives the displayed numbers.
The Euler numbers agree because a flop replaces smooth rational curves by smooth rational curves.

(3) Write $D=a\Hm+b\Fm=(a+2b)\Hp-b\Fp$, so that $D\cdot C_i=-b$.
By (1),
\[
D^3\big|_{X_+}=7(a+2b)^3-18(a+2b)^2b=7a^3+24a^2b+12ab^2-16b^3,
\]
while by (2) $D^3|_{X_-}=7a^3+24a^2b+12ab^2$; the difference is $16b^3=-16(D\cdot C_i)^3$.
Similarly $c_2(X_+)\cdot D=46(a+2b)-24b=46a+68b$ and $c_2(X_-)\cdot D=46a+36b$, whose difference is $-32b=32(D\cdot C_i)$.
These are the standard formulas for a simultaneous Atiyah flop of 16 curves \cite{LLW,LLQW}.
\end{proof}

The rational even cohomology rings are the Poincar\'e duality algebras determined by the cubic forms,
\begin{align}
\Rminus&=\QQ[\Hm,\Fm]/(4\Hm^2-8\Hm\Fm+9\Fm^2,\Fm^3),\label{eq:ringminus}\\
\Rplus&=\QQ[\Hp,\Fp]/(\Fp^2,\Hp^2(6\Hp-7\Fp)),\label{eq:ringplus}
\end{align}
with divisor classes in degree 1.
Indeed, the displayed quadratic relations are exactly the quadrics annihilated by multiplication with both divisor classes, e.g.\ $\Hm(4\Hm^2-8\Hm\Fm+9\Fm^2)=28-64+36=0$ and $\Fm(4\Hm^2-8\Hm\Fm+9\Fm^2)=32-32=0$.
Moreover the multiplication $\operatorname{Sym}^2H^2(X_\pm,\QQ)\to H^4(X_\pm,\QQ)$ is surjective: the pairings of $H^2,HF,F^2$ with $H,F$ (in the respective bases) form the matrices
$\bigl(\begin{smallmatrix}7&6\\6&0\\0&0\end{smallmatrix}\bigr)$ and
$\bigl(\begin{smallmatrix}7&8\\8&4\\4&0\end{smallmatrix}\bigr)$, both of
rank 2, and $H^4$ is dual to $H^2$.
Hence these rings are isomorphic to the full even cohomology rings, and both have Hilbert function $(1,2,2,1)$.

The projection of \eqref{eq:Xminus} to $\PP^2$ has general fiber a smooth $(2,2)$ complete intersection in $\PP^3$, hence a genus-one curve.
Indeed, over $y\in\PP^2$ the equation $\ell_1y_1+\ell_2y_2+\ell_3y_3=0$ is a hyperplane in $\PP^4$, and the two components of the $\mathcal Q$-valued section restrict to 2 quadrics in that $\PP^3$; smoothness of the general fiber follows from generic smoothness.

\begin{proposition}\label{prop:index4}
\[
\operatorname{Pic}(X_+)/\operatorname{tors}\simeq\mathbb Z\Hp\oplus\mathbb Z\Fp, \qquad
\operatorname{Pic}(X_-)/\operatorname{tors}\simeq\mathbb Z\Hm\oplus\mathbb Z\Fm.
\]
The genus-one fibration $X_-\to\PP^2$ has index exactly $4$.
In particular it has no rational section or bisection.
\end{proposition}

\begin{proof}
Since $h^{1,1}(X_+)=2$ and $h^{2,0}(X_+)=0$, the Picard rank is 2, so a class in $\operatorname{Pic}(X_+)/\operatorname{tors}$ can be written $D=r\Hp+s\Fp$ with $r,s\in\QQ$.
For a flopping curve $C_i$ one has $\Hp\cdot C_i=0$ and $\Fp\cdot C_i=1$, hence $s=D\cdot C_i\in\mathbb Z$.
Moreover $D\Hp^2=7r+6s$ and $D\Hp\Fp=6r$ are integers, so $r=(7r+6s)-6r-6s\in\mathbb Z$.
The change of basis $(\Hp,\Fp)\mapsto(\Hm,\Fm)=(\Hp,2\Hp-\Fp)$ is unimodular, which gives the statement for $X_-$.
The fiber class of $X_-\to\PP^2$ is $\Fm^2$, and
\[
(m\Hm+n\Fm)\cdot \Fm^2=4m .
\]
The divisor $\Hm$ gives a multisection of degree 4.
Conversely, the closure of a closed point of the generic fiber is a prime divisor on the smooth threefold $X_-$, hence Cartier, and its degree over the base is divisible by 4 by the preceding calculation.
Thus the index is exactly $4$.
\end{proof}


\section{The mirror family and its compactification}\label{sec:Bmodel}

Since the ambient space of $X_-$ is not Gorenstein (Section~\ref{sec:nefpartition}), the Batyrev--Borisov construction
\cite{BB} does not apply.  The Hori--Vafa equations single out a
2-parameter Laurent family (Lemma~\ref{lem:HVelimination}), which we take as the definition of the B-model (Definition~\ref{def:mirrorfamily}); all its properties are proved directly.

\subsection{The toric chamber \texorpdfstring{$\mathcal A_-$}{A-} and its nef decomposition}
\label{sec:nefpartition}

Let $\mathcal A_\pm=W/\!/_\vartheta \mathbb G$ for $\vartheta$ in the interior of $\mathfrak C_\pm$, with $W$, $\mathbb G$ and the weights of Section~\ref{sec:aux}.
These are projective simplicial toric eightfolds whose Cox coordinates are the 10 variables $u,v,x_0,\ldots,x_4,y_1,y_2,y_3$, with class group $\mathbb Z^2$ and the degrees
\[
\deg(u)=\deg(v)=(1,0),\qquad
\deg(x_i)=(0,1),\qquad
\deg(y_j)=(-1,2).
\]
By Propositions~\ref{prop:pluschamber} and~\ref{prop:minuschamber}, $X_\pm$ is the complete intersection in $\mathcal A_\pm$ of the 5 equations \eqref{eq:auxsystem}, whose classes are
\begin{equation}\label{eq:fiveclasses}
L_1=L_2=L_3=L_4=(0,2),\qquad L_5=(-1,3),
\end{equation}
and
\[
-K_{\mathcal A_\pm}=2(1,0)+5(0,1)+3(-1,2)=(-1,11)=L_1+\cdots+L_5 .
\]

\paragraph{Positivity in the two chambers.}
The nef cone of $\mathcal A_\pm$ is the closed chamber $\overline{\mathfrak C_\pm}$.
In the plus chamber, $\overline{\mathfrak C_+}=\operatorname{Cone}((1,0),(0,1))$ contains neither $-K=(-1,11)$ nor $L_5=(-1,3)$: the ambient space $\mathcal A_+$ is not $\QQ$-Fano, and the generalized equation is not a nef ambient divisor.
This is the toric reflection of the generalized nature of $X_+$.
In the minus chamber, $\overline{\mathfrak C_-}=\operatorname{Cone}((0,1),(-1,2))$ and
\[
-K=9(0,1)+(-1,2),\qquad L_5=(0,1)+(-1,2),\qquad (0,2)=2(0,1),
\]
so $-K$ and $L_5$ lie in the interior of $\mathfrak C_-$ and the classes $(0,2)$ lie on its boundary.
Since $\mathcal A_-$ is simplicial, every class is $\QQ$-Cartier; hence $-K_{\mathcal A_-}$ is ample, $\mathcal A_-$ is a toric $\QQ$-Fano variety, and all 5 classes \eqref{eq:fiveclasses} are nef.
In summary:
\[
\begin{array}{c|cc}
 & \mathfrak C_+ & \mathfrak C_-\\ \hline
-K=(-1,11) & \text{not nef} & \text{ample}\\
(0,2) & \text{nef} & \text{nef}\\
(-1,3) & \text{not nef} & \text{nef}\\
\text{A-model} & X_+\text{ (gCICY)} & X_-\text{ (nef zero locus)}
\end{array}
\]

\paragraph{A nef decomposition.}
A \emph{nef decomposition} of $-K_{\mathcal A_-}$ refines the equality $L_1+\cdots+L_5=-K$ equation by equation: the toric boundary divisors $D_\rho$ are partitioned into 5 parts $E_j$ with $\sum_{\rho\in E_j}D_\rho\sim L_j$.
Since
\[
(1,0)+(-1,2)=(0,2),\qquad (0,1)+(0,1)=(0,2),\qquad (-1,2)+(0,1)=(-1,3),
\]
one choice is
\begin{equation}\label{eq:nefblocks}
E_1=\{u,y_1\},\quad
E_2=\{v,y_2\},\quad
E_3=\{x_0,x_1\},\quad
E_4=\{x_2,x_3\},\quad
E_5=\{y_3,x_4\}.
\end{equation}
The choice is not canonical, since the $y_j$, respectively the $x_i$, have equal classes; only the list \eqref{eq:fiveclasses} matters.
In our choice every part consists of 2 boundary divisors, so the torus equations below are binomials.

\paragraph{Singularities and the failure of the Batyrev--Borisov hypotheses.}
By Section~\ref{sec:aux}, the only nontrivial stabilizers of the $\mathbb G$-action on the $\vartheta$-semistable locus are the groups $\mu_2$ along the stratum $\{x=0\}$, acting by $-1$ on the 5 transverse coordinates.
Hence $\mathcal A_-$ is smooth away from a 3-dimensional stratum of quotient singularities of type $\frac12(1,1,1,1,1)$.
Their age is $5/2>1$, so by the Reid--Tai criterion they are terminal; they are not Gorenstein, since the generator $-I_5$ has determinant $-1$.
A class $(a,b)$ is Cartier along this stratum exactly when the corresponding character $(\lambda_1,\lambda_2)\mapsto
\lambda_1^a\lambda_2^b$ is trivial on the stabilizer $\{(1,\pm1)\}$, i.e.\ when $b$
is even.
Thus the 4 classes $(0,2)$ are Cartier, whereas
\[
-K_{\mathcal A_-}=(-1,11)\quad\text{and}\quad L_5=(-1,3)
\]
are only $\QQ$-Cartier, of index 2.
By Proposition~\ref{prop:minuschamber}, $X_-$ does not meet the singular stratum, so $X_-$ lies in the smooth locus of $\mathcal A_-$.
Nevertheless the data $(\mathcal A_-;E_1,\ldots,E_5)$ is not a nef partition in the sense of Batyrev--Borisov: that notion requires a Gorenstein toric Fano variety and Cartier classes $L_j$, 
which is exactly what produces the reflexive Gorenstein cone and the dual nef partition
\cite{BB}.  Hence the Batyrev--Borisov construction does not apply to $X_-$.
In the terminology of \cite{KanII}, the configuration \eqref{eq:Xplusintro} has a $[3,-1]$~block over $\PP^4$ with $a=2$ and $r=1$, which violates the Gorenstein criterion $a\mid r$.

\subsection{Hori--Vafa equations and the Laurent family}
\label{sec:HV}

\paragraph{The Hori--Vafa prescription.}
For a toric variety with Cox coordinates indexed by rays $\rho$, weight matrix $\mathsf W=(\mathsf W_{a\rho})$ and a decomposition of the boundary into parts $E_1,\ldots,E_k$, 
the Hori--Vafa torus equations \cite{HV} introduce one nonzero variable $z_\rho$ for each ray and impose, after rescaling the coefficients,
\[
\sum_{\rho\in E_j}z_\rho=1\quad(1\le j\le k),\qquad
\prod_{\rho\in E_1\cup\cdots\cup E_k} z_\rho^{\mathsf W_{a\rho}}=\mathrm{const}\quad(\text{one for each row of }\mathsf W).
\]
The additive equations encode the complete-intersection equations, and the multiplicative ones encode the toric quotient and supply the complex parameters.
For a Cartier nef partition of a Gorenstein toric Fano variety these are the dense-torus equations of the Batyrev--Borisov mirror
\cite{BB,HV}.

\paragraph{The present equations.}
With columns ordered as $u,v,x_0,\ldots,x_4,y_1,y_2,y_3$ the weight matrix is
\[
\mathsf W=\begin{pmatrix}
1&1&0&0&0&0&0&-1&-1&-1\\
0&0&1&1&1&1&1& 2& 2& 2
\end{pmatrix}.
\]
Its first row is the character $q_U$, dual to $\Fp$, and its second row the character $q_X$, dual to $H$.
For the parts \eqref{eq:nefblocks} the prescription gives, in the torus variables denoted by the same letters as the Cox coordinates,
\begin{equation}\label{eq:HVadditive}
u+y_1=1,\quad v+y_2=1,\quad x_0+x_1=1,\quad
x_2+x_3=1,\quad y_3+x_4=1,
\end{equation}
\begin{equation}\label{eq:HVmult}
\frac{uv}{y_1y_2y_3}=w_2,\qquad
x_0x_1x_2x_3x_4\,(y_1y_2y_3)^2=w_1 ,
\end{equation}
so that $w_1$ is attached to $H$ and $w_2$ to $\Fp$.
The 10 torus variables are cut down by 7 equations to a threefold.
Since the mirror coordinates of the nef chamber $X_-$ should be attached to the basis $(\Hm,\Fm)=(H,2H-\Fp)$, we put
\begin{equation}\label{eq:zw}
z_1=w_1w_2^2,\qquad z_2=w_2^{-1};
\end{equation}
indeed, for a curve class $\beta$, $w_1^{H\cdot\beta}w_2^{\Fp\cdot\beta}= z_1^{\Hm\cdot\beta}z_2^{\Fm\cdot\beta}$ because $\Fm=2H-\Fp$.

\begin{lemma}\label{lem:HVelimination}
Put $a=u/y_1$, $b=v/y_2$, $c=x_0/x_1$ and $d=x_2/x_3$.
On a dense open subset, the equations \eqref{eq:HVadditive}--\eqref{eq:HVmult} are equivalent to the single equation
\begin{equation}
G:=1-z_2ab-z_1\frac{(1+a)^2(1+b)^2(1+c)^2(1+d)^2}{a^2b^2cd}=0
\label{eq:laurent}
\end{equation}
in the torus $(\CC^*)^4_{a,b,c,d}$, the remaining variables being determined by $a,b,c,d$.
\end{lemma}

\begin{proof}
The first 4 additive equations give
\begin{gather*}
y_1=\frac1{1+a},\quad u=\frac a{1+a},\quad
y_2=\frac1{1+b},\quad v=\frac b{1+b},\\
x_1=\frac1{1+c},\quad x_0=\frac c{1+c},\quad
x_3=\frac1{1+d},\quad x_2=\frac d{1+d}.
\end{gather*}
The first multiplicative relation becomes $ab/y_3=w_2$, i.e.\ $y_3=z_2ab$, and then the last additive equation gives $x_4=1-z_2ab$.
Substituting into the second multiplicative relation,
\[
w_1=\frac{c}{(1+c)^2}\,\frac{d}{(1+d)^2}\,(1-z_2ab)\, \frac{z_2^2a^2b^2}{(1+a)^2(1+b)^2},
\]
and with $w_1=z_1z_2^2$ this is $z_1(1+a)^2(1+b)^2(1+c)^2(1+d)^2=a^2b^2cd\,(1-z_2ab)$, which is
\eqref{eq:laurent}.  Conversely, a solution of \eqref{eq:laurent} with
$a,b,c,d\notin\{0,-1\}$ and $z_2ab\notin\{0,1\}$ determines all 10 variables by the displayed formulas.
\end{proof}

\begin{definition}\label{def:mirrorfamily}
The \emph{mirror family} is the 2-parameter family of affine hypersurfaces $\{G=0\}\subset(\CC^*)^4$ given by \eqref{eq:laurent}, with parameters $(z_1,z_2)\in(\CC^*)^2$, together with the residue form
\[
\Omega=\Res_{G=0}\frac{\mathrm{d}a}{a}\wedge\frac{\mathrm{d}b}{b}\wedge\frac{\mathrm{d}c}{c}\wedge
\frac{\mathrm{d}d}{d}\Big/G .
\]
Its smooth projective crepant model $Y=Y_{z_1,z_2}$ is constructed in Sections~\ref{sec:doublecover}--\ref{sec:crepant}.
\end{definition}

\paragraph{Genericity conditions.}
Throughout this section we fix $(z_1,z_2)$ in the dense Zariski-open subset $U\subset(\CC^*)^2$ defined by the following conditions:
\begin{enumerate}
\item[(G1)] $z_1z_2(z_2-1)\ne0$;
\item[(G2)] $(z_1,z_2)$ is not on the torus discriminant, i.e.\
$\{G=0\}\subset(\CC^*)^4$ is smooth; its equation $\Delta(z_1,z_2)=0$ is determined in Proposition~\ref{prop:discriminant};
\item[(G3)] the strict transform $\Dres$ of the residual discriminant curve
(Section~\ref{sec:crepant}) is smooth, and the discriminant components meet only at the 13 points listed in Section~\ref{sec:crepant}.
\end{enumerate}
Each condition is Zariski-open, and they hold at a general point.
Statements requiring \emph{very general} parameters are marked as such.

\subsection{Double-cover model}
\label{sec:doublecover}

\begin{proposition}\label{prop:doublecover}
For $(z_1,z_2)\in U$, the hypersurface \eqref{eq:laurent} is birational to a normal projective double cover $\pi:Y_0\to\mathcal P=(\PP^1)^3_{a,b,d}$, branched along a reduced divisor of degree $(4,4,4)$, with trivial dualizing sheaf and
\[
(h^0(\OO_{Y_0}),h^1(\OO_{Y_0}),h^2(\OO_{Y_0}),h^3(\OO_{Y_0}))=(1,0,0,1).
\]
Under this birational map, $\Omega=-\,\mathrm{d}a\wedge \mathrm{d}b\wedge \mathrm{d}d/\xi$, with $\xi$ as in \eqref{eq:affinedouble}.
\end{proposition}

\begin{proof}
Put
\[
\mathsf L:=1-z_2ab,\qquad
\kappa:=z_1(1+a)^2(1+b)^2(1+d)^2 .
\]
Multiplying \eqref{eq:laurent} by $a^2b^2cd$ gives $\widetilde G:=a^2b^2d\mathsf L\,c-\kappa(1+c)^2=0$, i.e.
\begin{equation}\label{eq:quadraticinc}
\kappa c^2+(2\kappa-a^2b^2d\mathsf L)c+\kappa=0 .
\end{equation}
Thus $(a,b,c,d)\mapsto(a,b,d)$ is generically of degree 2.
The discriminant of \eqref{eq:quadraticinc} is
\[
\Delta_c=(2\kappa-a^2b^2d\mathsf L)^2-4\kappa^2
=a^2b^2d\mathsf L\,\Psi,\qquad
\Psi:=a^2b^2d\mathsf L-4\kappa .
\]
Since $a,b$ are units on the torus, $(ab)^2$ is a square, and on the hypersurface the function
\[
\xi:=\frac{2\kappa c+2\kappa-a^2b^2d\mathsf L}{ab}
\]
satisfies, by \eqref{eq:quadraticinc},
\begin{equation}
\xi^2=d\mathsf L\Psi.\label{eq:affinedouble}
\end{equation}
Conversely, where $\kappa\ne0$, $c$ is recovered from $\xi$, so
\eqref{eq:affinedouble} has the same function field.

To compactify, write $a=a_1/a_0$, $b=b_1/b_0$, $d=d_1/d_0$ and put
\begin{align*}
\mathsf L_h&=a_0b_0-z_2a_1b_1,\\
\Psi_h&=a_1^2b_1^2\mathsf L_h\,d_0d_1
 -4z_1a_0b_0(a_0+a_1)^2(b_0+b_1)^2(d_0+d_1)^2 .
\end{align*}
The branch section $\Dbr=d_0d_1\mathsf L_h\Psi_h$ has degree $(4,4,4)$, and the double cover $Y_0$ is the hypersurface $\{\eta^2=\Dbr\}$ in the total space of $\OO_{\mathcal P}(2,2,2)$, with $\pi_*\OO_{Y_0}=\OO_{\mathcal P}\oplus
\OO_{\mathcal P}(-2,-2,-2)$.

\emph{Reducedness.}  For $z_2\ne0$ the form $\mathsf L_h$ is irreducible.
As a binary quadratic form in $(d_0,d_1)$, $\Psi_h$ has coefficients $-4T$, $S-8T$, $-4T$ with
\[
S=a_1^2b_1^2\mathsf L_h,\qquad T=z_1a_0b_0(a_0+a_1)^2(b_0+b_1)^2,
\]
which have no common factor, and discriminant $S(S-16T)$, whose square class $\mathsf L_h(S-16T)$ is not a square.
Hence $\Psi_h$ is irreducible, and $\Dbr$ is a product of 4 distinct irreducible factors, hence reduced.

\emph{Normality and the canonical class.}  A double cover
$\{\eta^2=\Dbr\}$ with reduced branch divisor is a hypersurface that is regular in codimension 1, hence normal.
The double-cover formula gives $\omega_{Y_0}=\pi^*(\omega_{\mathcal P}\otimes\OO_{\mathcal P}(2,2,2)) =\OO_{Y_0}$, and the cohomology follows from $\pi_*\OO_{Y_0}=\OO\oplus\OO(-2,-2,-2)$ and K\"unneth.

\emph{The residue.}  Since $1/(abcd\,G)=ab/\widetilde G$, we have
$\Omega=\Res_c\bigl(ab\,\mathrm{d}a\wedge \mathrm{d}b\wedge \mathrm{d}c\wedge \mathrm{d}d/\widetilde G\bigr) =ab\,\mathrm{d}a\wedge \mathrm{d}b\wedge \mathrm{d}d/(\partial\widetilde G/\partial c)$, and $\partial\widetilde G/\partial c=a^2b^2d\mathsf L-2\kappa(1+c)=-ab\,\xi$.
\end{proof}

The projection $Y_0\to\PP^1_a\times\PP^1_b$ exhibits $Y_0$ as a genus-one fibration: over a general point of $\PP^1_a\times\PP^1_b$ the branch divisor restricts to the 4 roots of $d_0d_1\Psi_h$ in $\PP^1_d$.
The next subsection puts this fibration into Weierstrass form.

\begin{remark}
The branch divisor is not an arrangement of smooth transverse components: $\Psi_h$ has special singularities and tangencies, so the existence of a crepant resolution cannot be read off from its degree.
Under a blowup of a smooth base center of codimension $r$ along which the branch divisor has multiplicity $m$, the normalized double cover changes its canonical class by $\bigl(r-1-\lfloor m/2\rfloor\bigr)E$; 
thus double or triple curves and quadruple or quintuple points are the relevant crepant centers in dimension 3.
Rather than controlling these centers directly (compare \cite{IL}), we pass to the Weierstrass model, where the resolution can be carried out explicitly.
\end{remark}

\subsection{Elliptic presentation and the Mordell--Weil group}
\label{sec:elliptic}
We now put the genus-one fibration $Y_0\to\PP^1_a\times\PP^1_b$ into Weierstrass form.
We keep $(z_1,z_2)\in U$.
Put $\mathcal B_0=\PP^1_a\times\PP^1_b$ and define the homogeneous sections
\[
 \mathsf L=a_0b_0-z_2a_1b_1,\qquad S=a_1^2b_1^2\mathsf L,\qquad
T=z_1a_0b_0(a_0+a_1)^2(b_0+b_1)^2,
\]
of bidegrees $(1,1)$, $(3,3)$ and $(3,3)$.
On the chart $a_0=b_0=1$ they restrict to $\mathsf L=1-z_2ab$ (the function of Section~\ref{sec:doublecover}), $S=a^2b^2\mathsf L$ and $T=z_1(1+a)^2(1+b)^2$, so that $\kappa=T(1+d)^2$ and $\Psi=S\,d-4T(1+d)^2$.

\begin{proposition}\label{prop:Weierstrass}
Over $\CC(\mathcal B_0)$ the Laurent model is birational to the elliptic curve
\begin{equation}
 y^2=x^3+\mathsf L(S-8T)x^2+16\mathsf L^2T^2x.\label{eq:Wmodel}
\end{equation}
It has the point of exact order 4
\[
P_4=(4\mathsf LT,\ 4a_1b_1\mathsf L^2T),\qquad 2P_4=(0,0).
\]
The associated projective Weierstrass model has fundamental line bundle $\mathscr L=\OO_{\mathcal B_0}(2,2)=-K_{\mathcal B_0}$, hence trivial dualizing sheaf.
\end{proposition}

\begin{proof}
Work on the chart $a_0=b_0=1$ and put $\lambda=T/S$ and $q=ab\mathsf L$, so that $q^2=\mathsf L S$.
By the formulas above,
\[
d\mathsf L\Psi=\mathsf LS\,d\{d-4\lambda(1+d)^2\}=q^2\,d\{d-4\lambda(1+d)^2\},
\]
so $\tilde y_0:=\xi/q$ turns \eqref{eq:affinedouble} into
\[
 \tilde y_0^2=d\{d-4\lambda(1+d)^2\}.
\]
The substitutions $\tilde x=-4\lambda d$, $\tilde y=-4\lambda\tilde y_0$ give
\begin{equation}
 E_\lambda:\quad \tilde y^2=\tilde x^3+(1-8\lambda)\tilde x^2+16\lambda^2\tilde x ,
 \label{eq:universal4}
\end{equation}
as one checks by substituting $d=-\tilde x/(4\lambda)$.
The scaling $x=q^2\tilde x$, $y=q^3\tilde y$ gives $y^2=x^3+(1-8\lambda)q^2x^2+16\lambda^2q^4x$, and $(1-8\lambda)q^2=\mathsf L(S-8T)$, $16\lambda^2q^4=16\mathsf L^2T^2$; this is \eqref{eq:Wmodel}.
For $x=4\mathsf LT$ the right-hand side of
\eqref{eq:Wmodel} equals $64\mathsf L^3T^3+16\mathsf L^3T^2(S-8T)
+64\mathsf L^3T^3=16\mathsf L^3T^2S=(4a_1b_1\mathsf L^2T)^2$, so $P_4$ lies on the curve. 
The tangent line at $P_4$ meets the curve again at $(0,0)$, which gives $2P_4=(0,0)$, a point of order 2; hence $P_4$ has order 4. 
Finally, the coefficients of $x^2$ and $x$ have bidegrees $(4,4)$ and $(8,8)$, i.e.\ they are sections of $\mathscr L^2$ and $\mathscr L^4$ for $\mathscr L=\OO(2,2)=-K_{\mathcal B_0}$, 
and the Weierstrass model with fundamental line bundle $-K$ of the base has trivial dualizing sheaf.
\end{proof}

\paragraph{Tate form.}
Put
\[
q=a_1b_1\mathsf L,\qquad \alpha=2q\in H^0(\mathcal B_0,\mathscr L),\qquad
 \beta=4\mathsf LT\in H^0(\mathcal B_0,\mathscr L^2).
\]
Since $S=q^2/\mathsf L$ one has $\mathsf L(S-8T)=\alpha^2/4-2\beta$ and $16\mathsf L^2T^2=\beta^2$, and $P_4=(\beta,\alpha\beta/2)$.
The change of coordinates
\[
x=X+\beta,\qquad y=Y+\tfrac{\alpha}{2}X+\tfrac{\alpha\beta}{2},
\]
which moves $P_4$ to the origin, transforms \eqref{eq:Wmodel} into the Tate form
\begin{equation}
 Y^2+\alpha XY+\alpha\beta Y=X^3+\beta X^2 ,
 \label{eq:Tate4}
\end{equation}
with Tate coefficients $(\mathsf a_1,\mathsf a_2,\mathsf a_3,\mathsf a_4,\mathsf a_6)=(\alpha,\beta,\alpha\beta,0,0)$.
The standard formulas give $b_2=\alpha^2+4\beta$, $b_4=\alpha^2\beta$, $b_6=\alpha^2\beta^2$, $b_8=\alpha^2\beta^3$, and hence
\begin{align}
 c_4&=\alpha^4-16\alpha^2\beta+16\beta^2,\notag\\
 \Delta_E&=\alpha^2\beta^4(\alpha^2-16\beta)
 =4096a_1^2b_1^2\mathsf L^7T^4(S-16T),\label{eq:ellDisc}
\end{align}
where the last equality uses $\alpha^2-16\beta=4\mathsf L(S-16T)$.

\begin{remark}
The universal elliptic surface \eqref{eq:universal4} over $\PP^1_\lambda$ has singular fibers $I_4$, $I_1$ and $I_1^*$ at $\lambda=0,1/16,\infty$.
Their Euler numbers add up to $4+1+7=12$, so its minimal smooth model is a rational elliptic surface.
Its trivial lattice $\mathrm U\oplus A_3\oplus D_5$, with $\mathrm U$ the hyperbolic plane, has rank $10=\rho$, so the Mordell--Weil rank is zero, and since the N\'eron--Severi lattice is unimodular, the torsion group has order $\sqrt{\det A_3\det D_5}=\sqrt{16}=4$.
Thus its Mordell--Weil group is $\mathbb Z/4\mathbb Z$, generated by the point of order 4.
This modular surface for $\Gamma_1(4)$ is standard; see \cite{HO}.
\end{remark}

\paragraph{The Mordell--Weil group for very general parameters.}
For every $(z_1,z_2)\in U$ the pullback of $E_\lambda$ along $\lambda=T/S$ contains the torsion group $\mathbb Z/4\mathbb Z$ generated by $P_4$.
Additional sections could a priori appear after pullback; this is excluded for very general parameters by a genus-2 curve and an arithmetic specialization.

\begin{proposition}\label{prop:MW}
For very general complex parameters $(z_1,z_2)$, the elliptic curve
\eqref{eq:Wmodel} over $\CC(a,b)$ has
\[
 \operatorname{MW}=\mathbb Z/4\mathbb Z.
\]
\end{proposition}
\begin{proof}
The idea is to fiber $\mathcal B_0$ by the level curves $C_\lambda$ of $\lambda=T/S$: a new section would give a nonconstant map from the genus-2 curve $C_\lambda$ to $E_\lambda$.
By specialization to $\mathbb F_7$, the Jacobian of $C_\lambda$ is ordinary while $E_\lambda$ specializes to a supersingular curve, so no such map exists.
The details are in Appendix~\ref{app:MW}.
\end{proof}

\subsection{Crepant compactification and topology}
\label{sec:crepant}
\paragraph{Base modifications.}
The Weierstrass model \eqref{eq:Tate4} over $\mathcal B_0$ is not minimal at finitely many points.
For the Tate form, a point is non-minimal, i.e.\ the short Weierstrass coefficients $f,g$ vanish there to orders at least $(4,6)$, exactly when $\operatorname{ord}\alpha\ge1$ and $\operatorname{ord}\beta\ge2$; a case analysis of the zeros of $\alpha=2a_1b_1\mathsf L$ and $\beta=4\mathsf LT$ then shows that for $(z_1,z_2)\in U$ the non-minimal points are the 6 points below (Lemma~\ref{lem:nonminimal} in Appendix~\ref{app:nonminimal}).
All points are specified by affine values $a=a_1/a_0$, $b=b_1/b_0$, with $0$ and $\infty$ in their projective meaning.
Blow up the 6 distinct points
\begin{align*}
 p_1&=(0,\infty), & p_2&=(\infty,0),\\
 p_3&=(0,-1), & p_4&=(-1,0),\\
 p_5&=(-1,-1/z_2), & p_6&=(-1/z_2,-1).
\end{align*}
At $p_5,p_6$ blow up once more, at the intersection of the first exceptional curve with the strict transform of $\mathsf L=0$.
Denote the resulting surface, obtained by 8 blowups, by $\mathcal B_8$.

\begin{proposition}\label{prop:8blowups}
For $(z_1,z_2)\in U$ the Weierstrass base modifications just described are crepant.
After them the fundamental line bundle is $-K_{\mathcal B_8}$ and there are no remaining $(4,6)$ points.
Moreover,
\[
 \rho(\mathcal B_8)=10,\qquad K_{\mathcal B_8}^2=0.
\]
\end{proposition}
\begin{proof}
For \eqref{eq:Tate4}, $f=-c_4/48$ and $g=-c_6/864$ are the short Weierstrass coefficients.
At a point with $\operatorname{ord}\alpha\ge1$, $\operatorname{ord}\beta\ge2$, they have orders at least $(4,6)$, and by Lemma~\ref{lem:nonminimal} the 6 points listed above are exactly the non-minimal points on $\mathcal B_0$.
Up to analytic units their local forms and orders are
\[
\begin{array}{c|c|c|c}
 &\alpha&\beta&(\operatorname{ord}f,\operatorname{ord}g, \operatorname{ord}\Delta_E)\\ \hline
 p_1,p_2&u(v-cu)&v(v-cu)&(4,6,14)\\
 p_3,p_4&u&v^2&(4,6,12)\\
 p_5,p_6&u&uv^2&(4,6,16)
\end{array}
\]
where $c\ne0$ and omitted factors are units.
A base blowup with exceptional parameter $\epsilon$ changes $\alpha,\beta$ to $\alpha/\epsilon,\beta/\epsilon^2$.
Equivalently the short coefficients change by $f/\epsilon^4,g/\epsilon^6$.
These are the crepant Weierstrass transforms; the corresponding weighted ambient blowup has weights $(2,3,1,1)$ and discrepancy zero.
At $p_5,p_6$, in the chart $u=\epsilon t,v=\epsilon$, the remaining special point has $\alpha'=t$, $\beta'=\epsilon t$ up to units, and again has orders $(4,6,12)$.
A second blowup removes it.
(After the 8 blowups the total space may still be singular in codimension 2 over the discriminant curves; this is addressed in Proposition~\ref{prop:globalresolution}.)

For completeness, let $A_0,A_\infty,A_{-1}$ and $B_0',B_\infty,B_{-1}$ be the strict transforms of the corresponding coordinate curves.
Write $\mathsf L'$ for the strict transform of $\mathsf L$, $E_1,E_2$ for the exceptional curves over $p_1,p_2$, and $E_5,E_6$ for the strict transforms of the first exceptional curves over $p_5,p_6$.
The transformed sections have divisors
\begin{align*}
 \operatorname{div}(\alpha')&=A_0+B_0'+\mathsf L'+E_1+E_2,\\
 \operatorname{div}(\beta')&=\mathsf L'+A_\infty+B_\infty
                +2A_{-1}+2B_{-1}+E_5+E_6.
\end{align*}
In these supports the only shared component is $\mathsf L'$.
It meets no other component of $\operatorname{div}(\beta')$.
The intersections away from $\mathsf L'$ where $\alpha'=\beta'=0$ have both orders 1.
Thus no $(4,6)$ points remain.
The short invariants cannot both vanish away from $\alpha'=\beta'=0$.
The line bundle remains anticanonical by the usual canonical divisor formula for a surface blowup.
\end{proof}

At the generic points of the discriminant components the fiber calculation is
\[
\begin{array}{c|c|c}
 \text{curves on }\mathcal B_8&\text{fiber}&\text{root rank contribution}\\ \hline
 A_0,B_0',E_1,E_2&I_2&4\\
 A_\infty,B_\infty,E_5,E_6&I_4&12\\
 A_{-1},B_{-1}&I_8&14\\
 \mathsf L'&I_1^*&5\\
 \text{strict transform of }S-16T=0&I_1&0
\end{array}
\]
The $I_4,I_8$ fibers are split: their multiplicative tangent slopes are squares, as is also visible in \eqref{eq:Tate4} at $\beta'=0$.
The $I_1^*$ fiber is split as well.
One way to see this is to specialize the rational point of order 4 at its generic point.
In residue characteristic zero the formal group and the additive identity component have no torsion, so the order-4 point injects into the component group.
For $I_1^*$ the geometric component group is cyclic of order 4, and its nonsplit form would have at most 2 rational points.
This rules out the nonsplit form.
For $I_2$, splitness does not change its rank-1 contribution.

\paragraph{Divisor lattice and residual genus-2 curve.}
We keep $(z_1,z_2)\in U$; in particular condition (G3) holds.
Let $\mathcal W_8$ denote the minimal Weierstrass transform over $\mathcal B_8$.

Write $\fa,\fb$ for the pullbacks of the two ruling classes on $\PP^1_a\times\PP^1_b$ (the classes of $\{a=\mathrm{const}\}$ and $\{b=\mathrm{const}\}$), 
and $e_1,\ldots,e_8$ for the \emph{total-transform} exceptional classes, with $e_7,e_8$ belonging to the infinitely near points above $p_5,p_6$.
Thus
\[
 \fa^2=\fb^2=0,\quad \fa\fb=1,\quad e_i e_j=-\delta_{ij},\quad
K=-2\fa-2\fb+\sum_{i=1}^8e_i.
\]
The following table gives all discriminant components.
In the last column, $\varepsilon$ is the Euler number of the resolved generic fiber.
\[
\begin{array}{c|c|c|c}
 \text{curve}&\text{class}&\text{fiber}&\varepsilon\\ \hline
 A_0&\fa-e_1-e_3&I_2&2\\
 B_0'&\fb-e_2-e_4&I_2&2\\
 E_1&e_1&I_2&2\\
 E_2&e_2&I_2&2\\
 A_\infty&\fa-e_2&I_4&4\\
 B_\infty&\fb-e_1&I_4&4\\
 E_5&e_5-e_7&I_4&4\\
 E_6&e_6-e_8&I_4&4\\
 A_{-1}&\fa-e_4-e_5&I_8&8\\
 B_{-1}&\fb-e_3-e_6&I_8&8\\
 \mathsf L'&\fa+\fb-e_1-e_2-e_5-e_6-e_7-e_8&I_1^*&7\\
 \Dres&3\fa+3\fb-e_1-e_2-2e_3-2e_4-e_5-e_6-e_7-e_8&I_1&1
\end{array}
\]
Here $\Dres$ is the strict transform of $S-16T=0$.
Writing $D_1,\ldots,D_{12}$ for the curves in the table, the weighted sum $\sum_{i=1}^{12}(\operatorname{ord}_{D_i}\Delta_E)\,[D_i]$ is $-12K$, an independent discriminant-class check.
All curves other than $\Dres$ are smooth rational curves.
The complete matrix of intersection numbers of these 12 curves, computed from the classes in the table, is given in Appendix~\ref{app:intersections}; it is used below to show that the 13 collision points account for all intersections.

\begin{proposition}\label{prop:residualgenus}
For general parameters, $\Dres$ is a smooth connected curve of genus 2.
It meets the rest of the discriminant at exactly 2 points.
At one of these it passes through $E_1\cap B_\infty$, and at the other through $E_2\cap A_\infty$.
Its intersection with $B_\infty$, respectively $A_\infty$, has multiplicity 2; its intersection with $E_1$, respectively $E_2$, is transverse.
\end{proposition}
\begin{proof}
The original residual curve has bidegree $(3,3)$ and 2 ordinary nodes at $p_3,p_4$, so its normalization has genus 2.
Its divisor class on $\mathcal B_8$ also gives $\Dres^2=4$ and $K_{\mathcal B_8}\Dres=-2$, hence adjunction gives the same genus.
The local equations at the two remaining incidences show that $\Dres$ is transverse to $E_1,E_2$ and tangent with multiplicity 2 to $B_\infty,A_\infty$, respectively.
Smoothness, connectedness, and the local tangency calculations are recorded in Appendix~\ref{app:residualcurve}.
\end{proof}

\paragraph{The 13 collision points.}
The 9 ordinary multiplicative crossings are
\[
\begin{array}{c|c|c}
 \text{pairs of curves}&\text{number}&\text{generic types}\\ \hline
 (A_0,B_0'),(A_0,E_1),(B_0',E_2)&3&I_2+I_2\\
 (A_\infty,B_\infty)&1&I_4+I_4\\
 (A_\infty,B_{-1}),(A_{-1},B_\infty), (A_{-1},E_5),(B_{-1},E_6)&4&I_4+I_8\\
 (A_{-1},B_{-1})&1&I_8+I_8
\end{array}
\]
There are also the two triple points
\[
q_1=E_1\cap B_\infty\cap\Dres,\qquad q_2=E_2\cap A_\infty\cap\Dres,
\]
and the two double points
\[
r_1=E_1\cap \mathsf L',\qquad r_2=E_2\cap \mathsf L'.
\]
By the intersection matrix of Appendix~\ref{app:intersections}, every positive intersection number between two distinct discriminant components is accounted for by these 13 points, 
counting the two tangencies of Proposition~\ref{prop:residualgenus} with multiplicity 2.
In particular, a tangency contributes one puncture to the topological Euler number, not two.

\paragraph{Triple points and additive--multiplicative crossings.}
At the 4 points $q_1,q_2,r_1,r_2$ the fiber geometry is not a product of two transversal one-parameter degenerations, and the resolved central fibers are \emph{not} the Kodaira fibers that a transverse one-parameter slice would suggest.
We record the outcome here and give the chart-by-chart computations in Appendix~\ref{app:local}.

At the triple points $q_1,q_2$ the transformed Tate coefficients $\alpha,\beta$ are transverse local coordinates $u,v$ and the local equation is
\begin{equation}
 F_A=y^2+uxy+uvy-x^3-vx^2=0,\label{eq:localA}
\end{equation}
with discriminant $u^2v^4(u^2-16v)$ (generic fiber types $I_2,I_4,I_1$ along the three branches).
Blowing up 3 curves (Appendix~\ref{app:triple}) gives a local smooth crepant resolution whose central fiber has 5 reduced rational components and
\begin{equation}
 e(F_{q_i})=6.\label{eq:eA}
\end{equation}
It is not the Kodaira $I_1^*$ fiber of Euler number 7 suggested by a transverse slice.

At the double points $r_1,r_2$ one can choose coordinates with $\beta=v$, $\alpha=uv$, and the local equation is
\begin{equation}
 F_B=y^2+uvxy+uv^2y-x^3-vx^2=0.\label{eq:localB}
\end{equation}
Blowing up 5 curves (Appendix~\ref{app:addmult}) gives a local smooth crepant resolution whose central fiber has 7 reduced rational components and
\begin{equation}
 e(F_{r_i})=8,\label{eq:eB}
\end{equation}
rather than the 8 components of the Kodaira $I_3^*$ fiber seen after resolving a one-parameter slice.
Away from $u=0$ the same construction reproduces the usual $I_1^*$ fiber along $\mathsf L'$.

\paragraph{Multiplicative crossings and globalization.}
At each of the 9 remaining points the cubic has a node with nondegenerate fiberwise Hessian.
The holomorphic Morse lemma with parameters gives
\begin{equation}
 pq=u^m v^n,\qquad (m,n)\in\{(2,2),(4,4),(4,8),(8,8)\},\label{eq:evenCA}
\end{equation}
up to a unit, which can be absorbed.
Crucially, both exponents are even.
Blowing up the smooth ambient curve $(p,q,u)=0$ gives, in the $u$-chart,
\[
p'q'=u^{m-2}v^n.
\]
In the $p$-chart of the same blowup, writing $q=pq'$, $u=pu_1$, one obtains $q'=p^{m-2}u_1^m v^n$, which is smooth, and similarly for the $q$-chart.
Repeat $m/2$ times, then use the analogous $v$-blowups $n/2$ times.
This is a smooth crepant resolution by curve blowups alone.
Its central fiber has Euler number
\begin{equation}
 1+2(m/2)+2(n/2-1)+1=m+n.\label{eq:eMult}
\end{equation}
The starting Euler number 1 is that of a nodal cubic.
All but the last of the conics over the central point are pairs of lines; the last is smooth.
No choice of a small resolution of an ordinary threefold node is required.

\begin{proposition}\label{prop:globalresolution}
For general parameters the Weierstrass model $\mathcal W_8$ has a smooth projective crepant resolution $Y\to\mathcal W_8$ whose elliptic fibration over $\mathcal B_8$ is flat, and $Y$ is a Calabi--Yau threefold.
\end{proposition}
The proof is given in Appendix~\ref{app:globalresolution}.

\begin{proposition}\label{prop:rho46}
For $(z_1,z_2)\in U$ the smooth projective crepant model $Y$ of Proposition~\ref{prop:globalresolution} satisfies
\[
h^{1,1}(Y)=\rho(Y)=46.
\]
\end{proposition}
\begin{proof}
Since $h^{2,0}(Y)=h^2(\OO_Y)=0$ by Proposition~\ref{prop:globalresolution}, $h^{1,1}(Y)=\rho(Y)$.
First let $(z_1,z_2)$ be very general.
By Proposition~\ref{prop:MW} the Mordell--Weil rank is zero; the fibration is flat by Proposition~\ref{prop:globalresolution}, so no divisor lies over a point of the base, and the Shioda--Tate--Wazir formula \cite{Wazir} gives
\[
 \rho(Y)=\rho(\mathcal B_8)+1+\sum_{i=1}^{12}(\text{number of fiber components over }D_i-1)
=10+1+4+12+14+5=46,
\]
the last 4 terms being the root-rank contributions of the table above.
For arbitrary $(z_1,z_2)\in U$, note that the curve blowups of Proposition~\ref{prop:globalresolution} are centered along closures of singular sections that are defined uniformly over $U$; 
hence the resolved models form a smooth projective family over $U$.
Hodge numbers are constant in such a family, so $h^{1,1}(Y)=46$ on all of $U$.
\end{proof}


\begin{proposition}\label{prop:euler88}
For $(z_1,z_2)\in U$ the smooth projective crepant model $Y$ satisfies
\[
e(Y)=88,\qquad (h^{1,1}(Y),h^{2,1}(Y))=(46,2).
\]
\end{proposition}
\begin{proof}
Let $D_i^\circ$ denote $D_i$ with its collision points removed, and $\varepsilon_i$ the Euler number in the table.
Additivity of compactly supported Euler characteristic gives
\[
e(Y)=\sum_{i=1}^{12}\varepsilon_i e_c(D_i^\circ) +\sum_{p\ \mathrm{collision}}e(F_p).
\]
The smooth-fiber locus contributes zero.
Equivalently, starting with the closed curves, the baseline is
\[
4\cdot2\cdot2+4\cdot4\cdot2+2\cdot8\cdot2+7\cdot2+1\cdot(-2)=92.
\]
At a multiplicative crossing \eqref{eq:eMult} gives no correction.
At each $q_i$ the correction is $6-(2+4+1)=-1$, and at each $r_i$ it is $8-(2+7)=-1$.
Hence $e(Y)=92-2-2=88$.

Proposition~\ref{prop:rho46} gives $h^{1,1}(Y)=46$.
Since $Y$ is a Calabi--Yau threefold, $e(Y)=2(h^{1,1}(Y)-h^{2,1}(Y))$, so $h^{2,1}(Y)=46-44=2$.
\end{proof}


\section{Periods and the two maximally unipotent boundary charts}\label{sec:periods}

Since Batyrev--Borisov duality is not available, the evidence that $Y$ is mirror to $X_\pm$ must come from the B-model itself.
In this section we compute the periods of $Y$, identify the two maximally unipotent boundary points of their Picard--Fuchs system, and show that the indicial algebras at these points are isomorphic to the rational even cohomology rings of $X_-$ and $X_+$.
The long derivations are collected in Appendix~\ref{app:pfcalc}.

Throughout, $\theta_{z_i}=z_i\partial_{z_i}$ denotes the Euler operator, and in a differential operator the coefficient functions stand to the left of the Euler operators.
We write
\[
G=1-z_2ab-z_1V,\qquad V(a,b,c,d)=\frac{(1+a)^2(1+b)^2(1+c)^2(1+d)^2}{a^2b^2cd},
\]
as in \eqref{eq:laurent}.

\subsection{The fundamental period and its Picard--Fuchs system}
\label{sec:fundperiod}

\paragraph{The torus period.}
Let $\mathbb T=\{|a|=|b|=|c|=|d|=1\}\subset(\CC^*)^4$ be the compact torus, oriented so that $(2\pi i)^{-4}\int_{\mathbb T}\frac{\mathrm{d}a}{a}\wedge\frac{\mathrm{d}b}{b}\wedge\frac{\mathrm{d}c}{c}\wedge\frac{\mathrm{d}d}{d}=1$.
On $\mathbb T$ one has $|z_2ab|=|z_2|$ and $|V|\le2^8$.
Hence, in the region
\begin{equation}\label{eq:Ominus}
 \mathcal O_-=\{(z_1,z_2)\mid |z_2|+256|z_1|<1\},
\end{equation}
$G$ has no zero on $\mathbb T$, and we may define
\[
\Phi_-(z_1,z_2)=\frac{1}{(2\pi i)^4}
\int_{\mathbb T}\frac{1}{G}\,\frac{\mathrm{d}a}{a}\wedge\frac{\mathrm{d}b}{b}\wedge\frac{\mathrm{d}c}{c}\wedge\frac{\mathrm{d}d}{d}.
\]
We write $z=(z_1,z_2)$ and abbreviate the origin $(z_1,z_2)=(0,0)$ by
$z=0$.  The subscript refers to this boundary point, which will be shown in Section~\ref{sec:charts} to be the large complex structure point attached to $X_-$.

\begin{lemma}\label{lem:torusperiod}
Let $(z_1,z_2)\in\mathcal O_-\cap U$, and let $Y^\circ=\{G=0\}\subset(\CC^*)^4$.
There is a 3-cycle $\gamma\subset Y^\circ$ such that
\[
 \Phi_-(z_1,z_2)=\frac{1}{(2\pi i)^3}\int_\gamma\Omega ,
\]
where $\Omega$ is the residue form of Definition~\ref{def:mirrorfamily}.
In particular $\Phi_-$ is a period of the holomorphic 3-form of the smooth crepant model $Y$.
\end{lemma}

\begin{proof}
Put $\omega=\frac{\mathrm{d}a}{a}\wedge\frac{\mathrm{d}b}{b}\wedge\frac{\mathrm{d}c}{c}\wedge\frac{\mathrm{d}d}{d}$.
For $R>0$ let $\mathbb T_R=\{|a|=|b|=R,\ |c|=|d|=1\}$.
On $\mathbb T_R$ one has $|z_1V|\le 256|z_1|(1+R)^4/R^4$, so for $R$ large $|z_2ab|=|z_2|R^2>1+|z_1V|$.
Thus $G$ has no zero on $\mathbb T_R$ and
\[
 \frac1G=-\sum_{k\ge0}\frac{(1-z_1V)^k}{(z_2ab)^{k+1}}
\]
converges uniformly there.
Every monomial of $V$ has $a$-exponent in $\{-2,-1,0\}$, so every term of this series has negative $a$-exponent and $\int_{\mathbb T_R}\omega/G=0$.

The tori $\mathbb T$ and $\mathbb T_R$ are homologous in $(\CC^*)^4$, so $[\mathbb T]-[\mathbb T_R]$ lies in the kernel of $H_4((\CC^*)^4\setminus Y^\circ)\to H_4((\CC^*)^4)$.
By condition (G2), $Y^\circ$ is a smooth hypersurface, and the Gysin sequence identifies this kernel with the image of the tube map $\tau:H_3(Y^\circ)\to H_4((\CC^*)^4\setminus Y^\circ)$.
Choose $\gamma$ with $\tau(\gamma)=[\mathbb T]-[\mathbb T_R]$.
The residue formula $\int_{\tau(\gamma)}\omega/G=2\pi i\int_\gamma\Res(\omega/G)$ gives the displayed identity.
Finally, $Y^\circ$ is an open subset of $Y$ and $\Omega$ extends to the nowhere-vanishing holomorphic 3-form on the crepant model constructed in Section~\ref{sec:crepant}.
\end{proof}

\begin{proposition}\label{prop:fundperiod}
On $\mathcal O_-$,
\begin{align}
\Phi_-(z_1,z_2)
&=\sum_{n\ge0}\sum_{m=0}^{2n}\binom{n+m}{m}\binom{2n}{m}^2
       \binom{2n}{n}^2 z_1^nz_2^m\notag\\
&=\sum_{n\geq0}\sum_{m=0}^{2n}
 \frac{(2n)!^4\,(n+m)!}{(n!)^5\,(m!)^3\,(2n-m)!^2}\,z_1^nz_2^m.\label{eq:fundperiod}
\end{align}
\end{proposition}

\begin{proof}
On $\mathbb T$ we have $|z_2ab+z_1V|<1$, so $1/G=\sum_{N\ge0}(z_2ab+z_1V)^N$ converges uniformly, and
\[
 \frac1G=\sum_{n,m\ge0}\binom{n+m}{m}z_1^nz_2^m\, \frac{(1+a)^{2n}(1+b)^{2n}(1+c)^{2n}(1+d)^{2n}}{a^{2n-m}b^{2n-m}c^nd^n}.
\]
The integral $(2\pi i)^{-4}\int_{\mathbb T}(\cdot)\,\omega$ of a Laurent series is its coefficient of $a^0b^0c^0d^0$.
In the displayed term this coefficient is the product of the coefficient of $a^{2n-m}$ in $(1+a)^{2n}$, which is $\binom{2n}{m}$ (and zero when $m>2n$), 
the same factor for $b$, and the coefficient $\binom{2n}{n}$ of $c^n$ in $(1+c)^{2n}$ and of $d^n$ in $(1+d)^{2n}$.
The second line is the same expression written in factorials.
\end{proof}

We write $c(n,m)$ for the coefficient of $z_1^nz_2^m$ in
\eqref{eq:fundperiod}, with $c(n,m)=0$ for $m>2n$.  Two features will be
used repeatedly.
First, for each $n$ the coefficient of $z_1^n$ is a polynomial of degree $2n$ in $z_2$; this is why the same series will also make sense in the plus chart.
Second, on the boundary line $z_2=0$,
\[
 \Phi_-(z_1,0)=\sum_{n\ge0}\binom{2n}{n}^2z_1^n,
\]
which is the classical elliptic period with singular point $z_1=1/16$.

\paragraph{The Picard--Fuchs system.}
Let
\begin{align}
\mathcal L_2={}&4\theta_{z_1}^2-8\theta_{z_1}\theta_{z_2}+9\theta_{z_2}^2
 -9z_2(2\theta_{z_1}-\theta_{z_2})^2\notag\\
&-16z_1(4\theta_{z_1}^2+4\theta_{z_1}\theta_{z_2}+3\theta_{z_2}^2+4\theta_{z_1}+2\theta_{z_2}+1)\notag\\
&-16z_1z_2(8\theta_{z_1}^2+5\theta_{z_1}\theta_{z_2}-3\theta_{z_2}^2+9\theta_{z_1}-2\theta_{z_2}+1),\label{eq:L2}\\
\mathcal L_3={}&\theta_{z_2}^3-z_2(\theta_{z_1}+\theta_{z_2}+1)(2\theta_{z_1}-\theta_{z_2})^2.\label{eq:L3}
\end{align}

\begin{proposition}\label{prop:PF}
$\mathcal L_2\Phi_-=\mathcal L_3\Phi_-=0$.
\end{proposition}

\begin{proof}
An operator $z_1^{i}z_2^{j}P(\theta_{z_1},\theta_{z_2})$ sends $z_1^nz_2^m$ to $P(n,m)z_1^{n+i}z_2^{m+j}$.
Hence $\mathcal L_3\Phi_-=0$ is the recurrence
\[
m^3\,c(n,m)=(n+m)(2n-m+1)^2\,c(n,m-1)\qquad(n,m\ge0),
\]
with $c(n,-1)=0$.
For $m\le2n$ this is the ratio of consecutive terms in
\eqref{eq:fundperiod}; for $m=2n+1$ both sides vanish, because of the factor
$(2n-m+1)^2$ on the right; for $m>2n+1$ both sides vanish trivially.
The equation $\mathcal L_2\Phi_-=0$ is likewise a 4-term recurrence among $c(n,m)$, $c(n,m-1)$, $c(n-1,m)$ and $c(n-1,m-1)$.
After division by $c(n,m)$ it becomes an identity of rational functions of $(n,m)$, which is written out and verified in Appendix~\ref{app:pfrec}.
The rows with $n=0$ or $m=0$ are checked in Appendix~\ref{app:pfboundary}.
In fact both equations are the degree-zero parts of Proposition~\ref{prop:minuschart}(1) below.
\end{proof}

The operator $\mathcal L_3$ is thus the $z_2$-recurrence of the coefficients.
The operator $\mathcal L_2$ was found by solving the linear equations for an operator of order 2 with the displayed support in $z$; only the verification in Appendix~\ref{app:pfrec} is needed below.
We call the system $\mathcal L_2u=\mathcal L_3u=0$ the \emph{Picard--Fuchs system} of the family.
Proposition~\ref{prop:minuschart} shows that its rank is 6, which is $b_3(Y)=2+2h^{2,1}(Y)$.

\paragraph{Singular locus and the torus discriminant.}
For an operator $\mathcal L=\sum_{\alpha\in\mathbb Z_{\ge0}^2,\ |\alpha|\le k}p_\alpha(z)\theta^\alpha$ of order $k$, 
let $\sigma(\mathcal L)=\sum_{|\alpha|=k}p_\alpha(z)\xi_1^{\alpha_1}\xi_2^{\alpha_2}$ be its principal symbol, a binary form of degree $k$ in the logarithmic cotangent variables $(\xi_1,\xi_2)$.
At a point of $(\CC^*)^2$ where the binary forms $\sigma(\mathcal L_2)$ and $\sigma(\mathcal L_3)$ have no common zero in $\PP^1$, 
the characteristic variety of the system is the zero section, and the system is an integrable connection of rank at most $\dim\CC[\xi_1,\xi_2]/(\sigma(\mathcal L_2),\sigma(\mathcal L_3))=2\cdot3=6$.
The locus where they have a common zero is the vanishing locus of their resultant.
A direct calculation (Appendix~\ref{app:pfsymbol}) gives
\[
 \operatorname{Res}\bigl(\sigma(\mathcal L_2),\sigma(\mathcal L_3)\bigr)
=-64(z_2-1)\Delta(z_1,z_2),
\]
where
\begin{align}
\Delta(z_1,z_2)={}&1-24z_1(2+34z_2+9z_2^2)\notag\\
&+48z_1^2(16-320z_2+808z_2^2-756z_2^3+243z_2^4) -4096z_1^3(1-z_2)^3.\label{eq:Delta}
\end{align}
So the singular locus of the system in $(\CC^*)^2$ is contained in $\{(z_2-1)\Delta=0\}$.
The next proposition identifies the factor $\Delta$ with the torus discriminant, that is, with the locus of parameters for which $Y^\circ=\{G=0\}\subset(\CC^*)^4$ is singular.

\begin{proposition}\label{prop:discriminant}
\begin{enumerate}
\item The hypersurface $Y^\circ$ is singular exactly when
\begin{equation}
z_1=\frac{t^4}{16(t+1)^3(t+3)},\qquad
z_2=\frac2{t^2(t+3)}\label{eq:param}
\end{equation}
for some $t\in\CC\setminus\{0,-1,-3\}$.
Its singular points are then $a=b=t$, $c=d=1$.
\item The polynomial $\Delta$ is irreducible, and \eqref{eq:param} is a
birational parametrization of the curve $\{\Delta=0\}$.
Consequently $\{\Delta=0\}\cap(\CC^*)^2$ is precisely the torus discriminant, and condition (G2) is equivalent to $\Delta(z_1,z_2)\ne0$.
\item At a general point of the torus discriminant, $Y^\circ$ has a single
ordinary double point.
At $(z_1,z_2)=(-1,1/2)$, that is $t=-2$, it has an $A_2$ singularity, and the two parameters give a miniversal deformation of it.
\end{enumerate}
\end{proposition}

\begin{proof}
(1) A point of $Y^\circ$ is singular iff $G$ and its 4 logarithmic derivatives vanish there.
If $V=0$ at such a point, then $a\partial_aG=-z_2ab\ne0$, a contradiction; hence $V\ne0$.
Since $c\partial_c\log V=(c-1)/(c+1)$, the equation $c\partial_cG=-z_1V(c-1)/(c+1)=0$ gives $c=1$, and similarly $d=1$.
Next,
\[
a\partial_aG=-z_2ab+\frac{2z_1V}{1+a},\qquad b\partial_bG=-z_2ab+\frac{2z_1V}{1+b},
\]
and their difference forces $a=b=:t$.
Then $z_2t^2=2z_1V/(1+t)$, and $G=0$ gives $1=z_2t^2+z_1V=z_1V(t+3)/(t+1)$.
Since $V(t,t,1,1)=16(1+t)^4/t^4$, solving for $z_1$ and $z_2$ gives
\eqref{eq:param}.  Conversely every such $t$ gives a singular point.  The
excluded values $t=0,-1,-3$ are exactly those for which \eqref{eq:param} does not define a point of $(\CC^*)^2$.

(2) Substituting \eqref{eq:param} into \eqref{eq:Delta} gives zero identically.
The map $t\mapsto z_2$ has degree 3, and $t\mapsto z_1$ has degree 4.
For a general value of $z_2$ the 3 preimages $t$ have distinct values of $z_1$ (Appendix~\ref{app:pfsymbol}), so \eqref{eq:param} is birational onto its image.
Therefore the irreducible equation of the image has degree 3 in $z_1$ and degree 4 in $z_2$.
Since $\Delta$ vanishes on the image and has the same bidegree $(3,4)$, $\Delta$ is a constant multiple of that equation.
The limits $t\to0,-1,-3,\infty$ leave $(\CC^*)^2$, so the image of $\CC\setminus\{0,-1,-3\}$ is all of $\{\Delta=0\}\cap(\CC^*)^2$.

(3) This is a local calculation with the logarithmic Hessian and the cubic jet at $t=-2$, given in Appendix~\ref{app:discriminantcalc}.
\end{proof}

Two boundary values of $\Delta$ are
\[
\Delta(z_1,0)=(1-16z_1)^3,\qquad
\Delta(z_1,1)=1-1080z_1-432z_1^2.
\]
The first recovers the singular point $z_1=1/16$ of $\Phi_-(z_1,0)$.
The second shows, in particular, that the factor $z_2-1$ appearing in the characteristic resultant is not a component of the torus discriminant.

\subsection{Cohomology-valued Frobenius series and the minus chart}
\label{sec:charts}

\paragraph{Exponentials of nilpotent classes.}
Let $\mathcal R$ be one of the rings $\Rminus$, $\Rplus$ of
\eqref{eq:ringminus}--\eqref{eq:ringplus}, graded with divisor classes in degree 1, and let
$h,f\in\mathcal R^1$.
These classes are nilpotent: every product of 4 classes of degree 1 vanishes.
For $z_1,z_2$ in a simply connected domain of $(\CC^*)^2$ with chosen branches of $\log z_1$ and $\log z_2$, we define
\begin{equation}\label{eq:zH}
 z_1^{h}z_2^{f}
 :=\exp(h\log z_1+f\log z_2)
 =\sum_{k=0}^{3}\frac{(h\log z_1+f\log z_2)^k}{k!}
 \in\mathcal R\otimes_\QQ\CC[\log z_1,\log z_2].
\end{equation}
The sum is finite because $\mathcal R$ vanishes in degree at least 4.
If $\varepsilon_0,\ldots,\varepsilon_5$ is a basis of $\mathcal R$, the coefficients of $z_1^hz_2^f$ in this basis are 6 scalar functions, 
each a polynomial in $\log z_1$, $\log z_2$ of degree equal to the cohomological degree of the corresponding basis vector.
The two basic properties are
\begin{equation}\label{eq:zHprops}
 \theta_{z_1}\bigl(z_1^hz_2^f\bigr)=h\,z_1^hz_2^f,\qquad
 \theta_{z_2}\bigl(z_1^hz_2^f\bigr)=f\,z_1^hz_2^f,
\end{equation}
and, for continuation along a positive loop around $z_1=0$, $z_1^hz_2^f\mapsto e^{2\pi i h}z_1^hz_2^f$ (similarly for $z_2$).
Thus $z_1^hz_2^f$ is the monomial $z_1^{\rho_1}z_2^{\rho_2}$ with the exponents replaced by nilpotent classes.
This is a compact form of the Frobenius method: expanding in the classes is the same as differentiating $z_1^{\rho_1}z_2^{\rho_2}$ with respect to $(\rho_1,\rho_2)$ at $\rho=0$.

Similarly, if $\phi(x,y)$ is holomorphic near $(0,0)$, then $\phi(h,f)\in
\mathcal R\otimes\CC$ denotes the Taylor polynomial of $\phi$ of degree at most
3 evaluated at $(h,f)$.
This assignment is a ring homomorphism.

\paragraph{The Gamma coefficient.}
Define
\[
C(x,y)=\frac{\Gamma(1+2x)^4\,\Gamma(1+x+y)} {\Gamma(1+x)^5\,\Gamma(1+y)^3\,\Gamma(1+2x-y)^2}.
\]
For integers $n,m\ge0$ one has $C(n,m)=c(n,m)$; when $m>2n$ this is zero, because $1/\Gamma$ vanishes at nonpositive integers.
For fixed $n,m\ge0$ the ratio
\[
r_{n,m}(x,y)=\frac{C(n+x,m+y)}{C(x,y)}
\]
is a rational function of $(x,y)$ which is holomorphic near $(0,0)$.
Indeed, each quotient $\Gamma(1+s+k)/\Gamma(1+s)$ with $k\ge0$ is a polynomial with value $k!$ at $s=0$, so all factors except the last are holomorphic near $0$.
In the last factor, with $u=2x-y$ and $k=m-2n$,
\[
 \frac{\Gamma(1+u)}{\Gamma(1+u+2n-m)}=
 \begin{cases}
 1/(1+u)_{2n-m}, & m\le 2n,\\
 u(u-1)\cdots(u-k+1), & m=2n+k>2n,
 \end{cases}
\]
where $(1+u)_j=(1+u)\cdots(j+u)$.
In particular, for $m>2n$ the ratio $r_{n,m}$ is divisible by $(2x-y)^2$; it vanishes at $(0,0)$ but not identically.

\paragraph{The minus Frobenius series.}
Using these conventions we define the $\Rminus$-valued series
\begin{equation}
\mathcal I_-=z_1^{\Hm}z_2^{\Fm}\,\widehat{\mathcal I}_-,\qquad
\widehat{\mathcal I}_-=\sum_{n,m\ge0}r_{n,m}(\Hm,\Fm)\,z_1^nz_2^m .
\label{eq:I}
\end{equation}
Here $z_1^{\Hm}z_2^{\Fm}$ is the exponential \eqref{eq:zH} and $r_{n,m}(\Hm,\Fm)$ is the Taylor polynomial of $r_{n,m}$, reduced in $\Rminus$.
We write the components in the basis $1,\Hm,\Fm,\Hm\Fm,\Fm^2,\Hm\Fm^2$ as
\begin{equation}\label{eq:varpi}
\mathcal I_-=\varpi_0+\varpi_H\Hm+\varpi_F\Fm
+\varpi_{HF}\Hm\Fm+\varpi_{F^2}\Fm^2+\varpi_{HF^2}\Hm\Fm^2 .
\end{equation}
Since $r_{n,m}(0,0)=c(n,m)$, the degree-zero component is $\varpi_0=\Phi_-$.
The terms with $m>2n$ do not contribute to $\Phi_-$, but they contribute to the higher components through the factor $(2\Hm-\Fm)^2$, which is nonzero in $\Rminus$.
They must be kept, and they are the origin of the fiber-class terms $Q_F^m$ in the genus-zero predictions of Section~\ref{sec:quantumseries}.

\begin{proposition}\label{prop:minuschart}
\begin{enumerate}
\item The series $\widehat{\mathcal I}_-$ converges for $|z_1|,|z_2|$
sufficiently small, and $\mathcal L_2\mathcal I_-=\mathcal L_3\mathcal I_-=0$ in $\Rminus\otimes\CC$.
\item The 6 components \eqref{eq:varpi} are linearly independent.  They form a
basis of local solutions near $z=0$, and the Picard--Fuchs system has rank 6 at every point of $(\CC^*)^2\setminus\{(z_2-1)\Delta=0\}$.
\item The system extends to a logarithmic connection of rank 6 on a
neighborhood of $z=0$, with poles along $z_1z_2=0$.
\item Identify the local solution space with $\Rminus^\vee$ by
$\lambda\mapsto s_\lambda=\lambda(\mathcal I_-)$.
The monodromy around $z_i=0$ is $\exp(2\pi iN_i)$, where $N_1$ and $N_2$ are the transposes of multiplication by $\Hm$ and $\Fm$: $N_1(s_\lambda)=s_{\lambda\circ\Hm}$ and $N_2(s_\lambda)=s_{\lambda\circ\Fm}$.
\end{enumerate}
\end{proposition}

\begin{proof}
(1) On the polydisc $|x|,|y|\le1/8$ one has $|s|\le3/8$ for every argument $s$ occurring in $r_{n,m}$, and hence $K_0^{-j}j!\le|(1+s)_j|\le K_0^jj!$ for a constant $K_0$.
For $m=2n+k>2n$ one has in addition $|u(u-1)\cdots(u-k+1)|\le2^kk!$.
Together with $(n+m)!\,(m-2n)!^2/m!^3\le2^{n+m}$, this gives $|r_{n,m}|\le K^{n+m}$ on the polydisc for a constant $K$.
The Cauchy estimates give the same bound for the Taylor coefficients of degree at most 3.
This gives convergence.
By \eqref{eq:zHprops}, an operator $z_1^iz_2^jP(\theta)$ sends $z_1^{\Hm}z_2^{\Fm}z_1^nz_2^m$ to $P(n+\Hm,m+\Fm)\,z_1^{\Hm}z_2^{\Fm}z_1^{n+i}z_2^{m+j}$.
Hence the equations are the recurrences of Proposition~\ref{prop:PF} with $(n,m)$ replaced by $(n+\Hm,m+\Fm)$ and $c$ replaced by $r_{\cdot,\cdot}(\Hm,\Fm)$.
For $n,m\ge1$ they hold because the rational identities of Appendix~\ref{app:pfrec} hold identically in $(x,y)$ and all terms are holomorphic at $(0,0)$.
The rows $n=0$ and $m=0$ need a separate check, since there the terms with a negative index are simply absent; this check uses the relations of $\Rminus$ and is carried out in Appendix~\ref{app:pfboundary}.

(2) Write $\mathcal I_-=z_1^{\Hm}z_2^{\Fm}(1+O(z))$.
If $\lambda(\mathcal I_-)=0$ for some $\lambda\in\Rminus^\vee$, then the part of lowest order in $z$ gives $\lambda(z_1^{\Hm}z_2^{\Fm})=0$ as a polynomial in $\log z_1$ and $\log z_2$, so $\lambda(\Hm^i\Fm^j)=0$ for all $i,j$ and $\lambda=0$.
Thus there are 6 independent solutions near $z=0$.
The rank is at most 6 away from $\{(z_2-1)\Delta=0\}$ by the symbol computation above.
Since the local solutions form a local system on the connected complement, the rank is 6 everywhere there.

(3) This is a finite reduction argument, given in Appendix~\ref{app:pfframe}.
The 6 operators
\[
1,\ \theta_{z_1},\ \theta_{z_2},\ \theta_{z_1}\theta_{z_2},\
 \theta_{z_2}^2,\ \theta_{z_1}\theta_{z_2}^2
\]
form a frame over the ring of holomorphic functions near $0$, and $\theta_{z_1}$ and $\theta_{z_2}$ act on it by holomorphic matrices.

(4) By \eqref{eq:I}, the continuation of $\mathcal I_-$ around $z_1=0$ is $e^{2\pi i\Hm}\mathcal I_-$, since $\widehat{\mathcal I}_-$ is single-valued.
Hence $s_\lambda\mapsto s_{\lambda\circ e^{2\pi i\Hm}}$.
The same argument applies to $z_2$.
\end{proof}

\paragraph{Indicial algebras.}
We now define the algebra that the boundary point $z=0$ attaches to the system.
Let $\mathcal O=\CC\{z_1,z_2\}$ and $\mathcal D=\mathcal O\langle\theta_{z_1},\theta_{z_2}\rangle$.
Every $P\in\mathcal D$ can be written uniquely as $P=P_0(\theta_{z_1},\theta_{z_2})+P'$, where $P_0$ is a polynomial and all coefficients of $P'$ lie in the maximal ideal $(z_1,z_2)$.
We call $P_0$ the \emph{indicial polynomial} of $P$.
It is characterized by
\[
P\bigl(z_1^{\rho_1}z_2^{\rho_2}(1+O(z))\bigr) =z_1^{\rho_1}z_2^{\rho_2}\bigl(P_0(\rho_1,\rho_2)+O(z)\bigr).
\]
For a left ideal $\mathfrak L\subset\mathcal D$, the set $I_{\mathrm{ind}}(\mathfrak L)=\{P_0\mid P\in\mathfrak L\}$ is an ideal of $\CC[\rho_1,\rho_2]$, called the \emph{indicial ideal}, and $A_{\mathrm{ind}}=\CC[\rho_1,\rho_2]/I_{\mathrm{ind}}(\mathfrak L)$ is the
\emph{indicial algebra}.  It is defined on the B-side from the asymptotics of
the differential equations; a priori it has nothing to do with an A-model cohomology ring.
Note that $I_{\mathrm{ind}}(\mathfrak L)$ contains the indicial polynomials of the generators of $\mathfrak L$, but in general it can be larger.

\begin{corollary}\label{cor:Aindminus}
Let $\mathfrak L_-\subset\mathcal D$ be the left ideal generated by $\mathcal L_2$ and $\mathcal L_3$.
Then
\[
I_{\mathrm{ind}}(\mathfrak L_-)=(4\rho_1^2-8\rho_1\rho_2+9\rho_2^2,\ \rho_2^3),
\]
and the assignment $\rho_1\mapsto\Hm$, $\rho_2\mapsto\Fm$ induces an isomorphism
\begin{equation}\label{eq:Aindminus}
 A_{\mathrm{ind},-}\ \xrightarrow{\ \sim\ }\ \Rminus\otimes\CC
 \simeq H^{\mathrm{even}}(X_-,\CC).
\end{equation}
\end{corollary}

\begin{proof}
The indicial polynomials of $\mathcal L_2$ and $\mathcal L_3$ are $4\rho_1^2-8\rho_1\rho_2+9\rho_2^2$ and $\rho_2^3$, so the displayed ideal $I_0$ is contained in $I_{\mathrm{ind}}(\mathfrak L_-)$.
Conversely, let $P\in\mathfrak L_-$.
By Proposition~\ref{prop:minuschart}(1), $P\mathcal I_-=0$, and the lowest-order term gives $P_0(\Hm,\Fm)=0$ in $\Rminus$.
By
\eqref{eq:ringminus}, the kernel of $\CC[\rho_1,\rho_2]\to\Rminus\otimes\CC$ is
$I_0$.
Hence $P_0\in I_0$.
\end{proof}

The ring $\Rminus$ was computed in Section~\ref{sec:intersection} from the intersection theory of $X_-$.
The corollary computes the same ring from the differential equations alone.
Such identifications are familiar at large-radius limits in toric mirror symmetry and in GKZ systems
\cite{HKTY,Iritani,IPeriods}, but they are not automatic for an arbitrary
differential system.
Here the isomorphism is a direct calculation, and it is one of the checks that $z=0$ has the expected A-model interpretation.

\subsection{The plus chart}
\label{sec:pluschart}

\paragraph{Coordinates.}
Put
\[
z_1=w_1w_2^2,\qquad z_2=w_2^{-1},\qquad
 \thwone=w_1\partial_{w_1},\quad \thwtwo=w_2\partial_{w_2}.
\]
We write $w=(w_1,w_2)$ and abbreviate the origin
$(w_1,w_2)=(0,0)$ by $w=0$.
Then $\theta_{z_1}=\thwone$ and $\theta_{z_2}=2\thwone-\thwtwo$.
This change of coordinates is dual to the change of divisor basis $(\Hm,\Fm)=(\Hp,2\Hp-\Fp)$ of Section~\ref{sec:Amodels}.
Indeed, with the convention \eqref{eq:zH},
\begin{equation}\label{eq:zwmonomial}
 z_1^{\Hm}z_2^{\Fm}=(w_1w_2^2)^{\Hp}\,w_2^{-(2\Hp-\Fp)}=w_1^{\Hp}w_2^{\Fp}.
\end{equation}
The flopping curves have $\Fm\cdot C_i'=1$ on $X_-$ and $\Fp\cdot C_i=1$ on $X_+$; thus $z_2$ and $w_2=z_2^{-1}$ are the coordinates attached to the flopping class in the two charts, and $w_1=z_1z_2^2$ is attached to $\Hp$.

\paragraph{The operators.}
Substituting into $\mathcal L_2$ and multiplying by $w_2$ on the left gives
\begin{equation}
\mathcal D_2:=w_2\mathcal L_2
=-9\thwtwo^2+w_2\mathsf Q+w_1w_2^2\mathsf R+w_1w_2^3\mathsf S,\label{eq:D2}
\end{equation}
where
\begin{align*}
\mathsf Q&=24\thwone^2-28\thwone\thwtwo+9\thwtwo^2,\\
\mathsf R&=-16(6\thwone^2+7\thwone\thwtwo-3\thwtwo^2+5\thwone+2\thwtwo+1),\\
\mathsf S&=-16(24\thwone^2-16\thwone\thwtwo+3\thwtwo^2+8\thwone-2\thwtwo+1).
\end{align*}
Indeed, $2\theta_{z_1}-\theta_{z_2}=\thwtwo$, so the term $-9z_2(2\theta_{z_1}-\theta_{z_2})^2$ becomes $-9\thwtwo^2$ after multiplication by $w_2$.
The polynomials $\mathsf Q$, $\mathsf S$, $\mathsf R$ are the indicial part of $\mathcal L_2$ and the coefficients of $z_1$ and $z_1z_2$, rewritten in $(\thwone,\thwtwo)$.
Similarly $\widetilde{\mathcal L}_3:=w_2\mathcal L_3 =w_2(2\thwone-\thwtwo)^3-(3\thwone-\thwtwo+1)\thwtwo^2$.
Its indicial part is divisible by $\thwtwo^2$ and does not by itself cut out a finite-dimensional algebra.
We therefore replace it by the third-order operator
\begin{align}
\mathcal D_3={}&\thwone\mathsf Q+9w_1(3\thwone-\thwtwo+2)\mathsf R\notag\\
&+w_1w_2\{9(3\thwone-\thwtwo+1)\mathsf S+(\thwone+1)\mathsf R\}
 +w_1w_2^2(\thwone+1)\mathsf S.\label{eq:D3}
\end{align}
It is obtained by eliminating the term of lowest order in $w_2$:
\[
 \mathcal K=w_2^{-1}\bigl\{(3\thwone-\thwtwo+1)\mathcal D_2-9\widetilde{\mathcal L}_3\bigr\},\qquad
 \mathcal D_3=w_2^{-1}\bigl(9\mathcal K+\thwone\mathcal D_2\bigr),
\]
where both divisions by $w_2$ are exact in the ring of operators, after the Ore commutation rule $\thwtwo w_2=w_2(\thwtwo+1)$ is applied (Appendix~\ref{app:pfplus}).
Conversely, $9\widetilde{\mathcal L}_3=(3\thwone-\thwtwo+1)\mathcal D_2-w_2\mathcal K$ and $9\mathcal K=w_2\mathcal D_3-\thwone\mathcal D_2$.
Hence $\mathcal D_2$ and $\mathcal D_3$ generate the same left ideal as $\mathcal L_2$ and $\mathcal L_3$ wherever $w_2$ is invertible.
In particular they define the same local system on the complement of the singular locus.
We use the ideal $\mathfrak L_+$ generated by $\mathcal D_2,\mathcal D_3$ in $\CC\{w_1,w_2\}\langle\thwone,\thwtwo\rangle$ to describe the boundary point $w=0$.

\begin{proposition}\label{prop:pluschart}
\begin{enumerate}
\item There is a unique $\Rplus$-valued series
\[
 \mathcal I_+=w_1^{\Hp}w_2^{\Fp}\,\widehat{\mathcal I}_+,\qquad
 \widehat{\mathcal I}_+=\sum_{n,k\ge0}a_{n,k}\,w_1^nw_2^k,\qquad
a_{n,k}\in\Rplus,\ a_{0,0}=1,
\]
annihilated by $\mathcal D_2$ and $\mathcal D_3$.
Its degree-zero component is the re-expansion of $\Phi_-$ in $(w_1,w_2)$.
\item The components of $\mathcal I_+$ form a basis of local solutions near
$w=0$.
The system extends to a logarithmic connection of rank 6 there, and with the identification $\lambda\mapsto\lambda(\mathcal I_+)$ of the solution space with $\Rplus^\vee$, 
its monodromy logarithms are the transposes of multiplication by $\Hp$ and $\Fp$.
\item The indicial ideal is
$I_{\mathrm{ind}}(\mathfrak L_+)=(\varrho_2^2,\ \varrho_1^2(6\varrho_1-7\varrho_2))$, and $\varrho_1\mapsto\Hp$, $\varrho_2\mapsto\Fp$ induces an isomorphism
\[
A_{\mathrm{ind},+}\ \xrightarrow{\ \sim\ }\ \Rplus\otimes\CC\simeq H^{\mathrm{even}}(X_+,\CC).
\]
\end{enumerate}
\end{proposition}

\begin{proof}
We prove the three assertions in the order (2), (3), (1).

(2) Appendix~\ref{app:pfframe} gives an explicit logarithmic frame
\[
1,\thwone,\thwtwo,\thwone^2,\thwone\thwtwo,\thwone^2\thwtwo
\]
for the left ideal generated by $\mathcal D_2,\mathcal D_3$ near $w=0$.
Indeed, the determinant of the reduction matrix at the origin is $2^9\cdot3^{15}\ne0$.
Hence every derivative reduces holomorphically to this 6-element frame, so the system extends to a logarithmic connection of rank 6 at $w=0$.
Its residue operators commute, and they are nilpotent because they satisfy the homogeneous indicial relations of (3) below: $N_2^2=0$ and $N_1^2(6N_1-7N_2)=0$, whence $N_1^4=0$.
Once (1) is established, the identification $\lambda\mapsto\lambda(\mathcal I_+)$ identifies these residues with the transposes of multiplication by $\Hp$ and $\Fp$, by
\eqref{eq:zHprops}.

(3) The indicial polynomials of $\mathcal D_2$ and $\mathcal D_3$ are $-9\varrho_2^2$ and
\[
\varrho_1\mathsf Q(\varrho_1,\varrho_2)
=24\varrho_1^3-28\varrho_1^2\varrho_2+9\varrho_1\varrho_2^2.
\]
Modulo $\varrho_2^2$ the second polynomial is $4\varrho_1^2(6\varrho_1-7\varrho_2)$.
Thus
\[
(\varrho_2^2,\ \varrho_1^2(6\varrho_1-7\varrho_2))
 \subset I_{\mathrm{ind}}(\mathfrak L_+).
\]
The quotient by the ideal on the left has dimension 6, with basis $1,\varrho_1,\varrho_2,\varrho_1^2,\varrho_1\varrho_2, \varrho_1^2\varrho_2$.  On the other hand, let
$M=\CC\{w_1,w_2\}\langle\thwone,\thwtwo\rangle/\mathfrak L_+$ and let $\mathfrak m=(w_1,w_2)$.
Since every operator is its indicial polynomial plus an element of $\mathfrak m\mathcal D$, the map $P\mapsto P_0$ identifies $\CC[\varrho_1,\varrho_2]/I_{\mathrm{ind}}(\mathfrak L_+)$ with $M/\mathfrak mM$.
The frame just constructed generates $M$ by 6 elements, and $M$ has generic rank 6 by Proposition~\ref{prop:minuschart}(2); hence $M$ is free of rank 6 and $\dim M/\mathfrak mM=6$.
Therefore the two ideals coincide.
Under $\varrho_1\mapsto\Hp$, $\varrho_2\mapsto\Fp$, the resulting quotient is precisely $\Rplus\otimes\CC$ by \eqref{eq:ringplus}.

(1) By \eqref{eq:zHprops}, the equations $\mathcal D_2\mathcal I_+=0$ and $\mathcal D_3\mathcal I_+=0$ give recurrences for $a_{n,k}$.
The row $(n,k)$ of $\mathcal D_2$ reads
\[
9(k+\Fp)^2a_{n,k}=\mathsf Q(n+\Hp,k-1+\Fp)a_{n,k-1} +\mathsf R(\cdots)a_{n-1,k-2}+\mathsf S(\cdots)a_{n-1,k-3},
\]
where the arguments of $\mathsf R$ and $\mathsf S$ are shifted in the same way.
For $k\ge1$ the element $(k+\Fp)^2$ is invertible in $\Rplus$, so this determines $a_{n,k}$ from earlier terms.
For $k=0$ and $n\ge1$ the row of $\mathcal D_3$ determines $a_{n,0}$, because the leading factor $(n+\Hp)\mathsf Q(n+\Hp,\Fp)$ has constant term $24n^3\ne0$.
Hence a normalized Frobenius series, if it exists, is unique.

Existence now follows from the regular-singular Frobenius construction applied to the logarithmic rank-6 connection in (2), 
whose indicial algebra is the 6-dimensional algebra identified in (3): the commuting nilpotent residues act through multiplication by the two degree-1 generators, 
and the normalized exponentiated solution gives the required $\Rplus$-valued series (compare \cite{Deligne,HKTY,SST}).
Thus no separate compatibility assumption on the two recurrences is needed.

Finally, reduction modulo positive-degree classes is a ring homomorphism $\Rplus\to\QQ$, so the degree-zero part is the unique scalar solution with constant term 1.
The re-expanded series $\sum_{n\ge0}\sum_{k=0}^{2n}c(n,2n-k)w_1^nw_2^k$ converges near $w=0$, because the coefficient of $z_1^n$ in $\Phi_-$ is a polynomial of degree $2n$ in $z_2$.
It is annihilated by $\mathcal D_2=w_2\mathcal L_2$ and by $\mathcal D_3$, which lies in the ideal generated by $\mathcal L_2$ and $\mathcal L_3$ after $w_2$ is inverted.
Hence it is the degree-zero component of $\mathcal I_+$.
The monodromy statement in (2) now follows from
\eqref{eq:zHprops}, completing the proof.
\end{proof}

\paragraph{Maximal unipotency.}
Let $(N_1,N_2)$ be the monodromy logarithms at either boundary point, and let $N=s_1N_1+s_2N_2$ with $s_1,s_2>0$.
By Propositions~\ref{prop:minuschart}(4) and~\ref{prop:pluschart}(2), the ranks of the powers of $N$ are those of multiplication by $D=s_1H_\pm+s_2F_\pm$ on $H^{\mathrm{even}}(X_\pm,\CC)$.
The class $D$ is ample, being the restriction of a character in the interior of the GIT chamber.
By hard Lefschetz, $D\colon H^2\to H^4$ is an isomorphism, and $D^3\ne0$; indeed
\[
D^3=7s_1^3+24s_1^2s_2+12s_1s_2^2\ \text{ on }X_-,\qquad D^3=7s_1^3+18s_1^2s_2\ \text{ on }X_+ .
\]
Hence the ranks of $N^0,N,N^2,N^3,N^4$ are $6,4,2,1,0$, so $N$ has one Jordan block of size 4 and one of size 2.
Thus $z=0$ and $w=0$ are two maximally unipotent boundary points of the rank-6 system.
On the extremal rays the two sides differ: $\Fp^2=0$, whereas $\Fm^2\ne0$ and $\Fm^3=0$.
This reflects the K3 fibration of $X_+$ and the genus-one fibration of $X_-$ (Section~\ref{sec:Amodels}).

\section{Yukawa couplings and genus-zero tests}\label{sec:quantum}

\subsection{Yukawa couplings, mirror coordinates, and the prepotential}
\label{sec:yukawa}
This section concerns the explicit differential module of Section~\ref{sec:periods}.
We use the normalization \(\int_{X_-}\Hm\Fm^2=4\).
The strategy of extracting Yukawa functions and a prepotential from a Picard--Fuchs system is classical; compare
\cite{HKTY}. The specific formulas below follow by rational reduction of
\eqref{eq:L2}--\eqref{eq:L3}.

Let \(\mathbf e_0\) be the cyclic generator and set
\[
 \mathbf e=(\mathbf e_0,\theta_{z_1} \mathbf e_0,\theta_{z_2} \mathbf e_0,\theta_{z_1}\theta_{z_2} \mathbf e_0, \theta_{z_2}^2\mathbf e_0,\theta_{z_1}\theta_{z_2}^2\mathbf e_0)^t
=(\mathbf e_0,\dots,\mathbf e_5)^t.
\]
Reduce derivatives of total order at most 4 by the 9 relations
\[
 \mathcal L_2,\ \mathcal L_3,\ \theta_{z_1}\mathcal L_2,\ \theta_{z_2}\mathcal L_2, \ \theta_{z_1}^2\mathcal L_2,\ \theta_{z_1}\theta_{z_2}\mathcal L_2,\ \theta_{z_2}^2\mathcal L_2, \ \theta_{z_1}\mathcal L_3,\ \theta_{z_2}\mathcal L_3.
\]
Their coefficients form a $9\times9$ system on the monomials outside the 6 displayed basis elements.
Inverting this system over \(\QQ(z_1,z_2)\) defines rational matrices \(A_1,A_2\) by
\(\theta_{z_1} \mathbf e=A_1\mathbf e\), \(\theta_{z_2} \mathbf e=A_2\mathbf e\).  Direct rational reduction gives
\[
 \theta_{z_1} A_2-\theta_{z_2} A_1+A_2A_1-A_1A_2=0.
\]

Define the polynomials
\begin{align*}
 \mathsf N_{111}&=7+z_1(448-1024z_2+540z_2^2)+256z_1^2(1-z_2)^2,\\
 \mathsf N_{112}&=1-z_1(8+202z_2-216z_2^2)-32z_1^2(z_2-1)(z_2-4),\\
 \mathsf N_{122}&=1+z_1(-32+272z_2+1188z_2^2)+64z_1^2(z_2+2)^2,\\
 \mathsf N_{222}&=1+z_1(136-460z_2-756z_2^2)+16z_1^2(z_2^2-20z_2-8).
\end{align*}

\begin{proposition}\label{prop:yukawa}
The rational differential module admits a nondegenerate horizontal alternating form whose logarithmic Yukawa functions, normalized at the minus origin, are
\begin{equation}
C_{111}=\frac{\mathsf N_{111}}\Delta,\qquad
 C_{112}=\frac{8\mathsf N_{112}}\Delta,\qquad
 C_{122}=\frac{4\mathsf N_{122}}\Delta,\qquad
 C_{222}=\frac{16z_2\mathsf N_{222}}{(1-z_2)\Delta}.
 \label{eq:Yukawas}
\end{equation}
In particular their limits along \(z_1=0\) are
\begin{equation}
 (C_{111},C_{112},C_{122},C_{222})
       =\left(7,8,4,\frac{16z_2}{1-z_2}\right).
 \label{eq:YukawaBoundary}
\end{equation}
\end{proposition}
\begin{proof}
An explicit rational certificate for the horizontal pairing is given in Appendix~\ref{app:yukawa}.
\end{proof}

\paragraph{Mirror coordinates and B-model predictions.}
\label{sec:quantumseries}
Write the series $\widehat{\mathcal I}_-$ of \eqref{eq:I}, that is, the Frobenius series with its exponential factor removed, as
\[
 \Pi_0+\Pi_H\Hm+\Pi_F\Fm+\Pi_{HF}\Hm\Fm+\Pi_{F^2}\Fm^2+\Pi_{HF^2}\Hm\Fm^2.
\]
All terms with \(m>2n\) must still be kept.
Define the normalized coefficients
\[
m_H=\frac{\Pi_H}{\Pi_0},\quad m_F=\frac{\Pi_F}{\Pi_0},\quad m_{HF}=\frac{\Pi_{HF}}{\Pi_0},\quad m_{F^2}=\frac{\Pi_{F^2}}{\Pi_0},\quad m_{HF^2}=\frac{\Pi_{HF^2}}{\Pi_0},
\]
and the flat coordinates
\[
t_H=\log z_1+m_H,\qquad t_F=\log z_2+m_F.
\]
With \(Q_H=e^{t_H}\) and \(Q_F=e^{t_F}\), the mirror map begins
\begin{align*}
 m_H&=8z_1+52z_1^2+144z_1z_2+\frac{1472}{3}z_1^3+5504z_1^2z_2+106z_1z_2^2+O_4,\\
 m_F&=16z_1+140z_1^2+16z_1z_2+\frac{4672}{3}z_1^3+3744z_1^2z_2-32z_1z_2^2+O_4,
\end{align*}
where \(O_4\) is total degree at least 4.

Set
\[
 \mathsf U=m_{HF}-m_H^2-m_Hm_F,\qquad \mathsf V=m_{F^2}-\frac12m_F^2+\frac98m_H^2,
\]
\begin{equation}
 \mathcal F_{\mathrm{inst}}=-2\left(m_{HF^2}-(2m_H+m_F)m_{HF}-m_Hm_{F^2}+\frac7{12}m_H^3+2m_H^2m_F+m_Hm_F^2\right),
 \label{eq:fexplicit}
\end{equation}
and express all of these as functions of \(Q_H,Q_F\) by the inverse mirror map.
The isotropy relations associated with Proposition~\ref{prop:yukawa} give
\[
 \partial_{t_H}\mathcal F_{\mathrm{inst}}=8\mathsf U+4\mathsf V,\qquad \partial_{t_F}\mathcal F_{\mathrm{inst}}=4\mathsf U.
\]
They follow from the expansion
\(\Pi_0^{-1}\mathcal I_-=e^{\Hm t_H+\Fm t_F}(1+\mathsf U\Hm\Fm+\mathsf V\Fm^2+\hat r\,\Hm\Fm^2)\), where \(\hat r\) is the bracket in \eqref{eq:fexplicit}, together with the isotropy of this normalized Frobenius vector for the horizontal pairing.
The point-class component of the normalized Frobenius vector is
\(-2\mathcal F_{\mathrm{inst}}\) in the sign convention of
\eqref{eq:prepotential} below, and \(\Hm\Fm^2=4\,\mathrm{pt}\) since
\(\int\Hm\Fm^2=4\); this gives \(\mathcal F_{\mathrm{inst}}=-2\hat r\).

Thus the formal prepotential is
\begin{equation}
 \mathcal F=\frac16(7t_H^3+24t_H^2t_F+12t_Ht_F^2)+\mathcal F_{\mathrm{inst}}(Q_H,Q_F),
 \label{eq:prepotential}
\end{equation}
up to the usual polynomial terms of degree at most 2, which are not fixed by the differential module alone.
The corresponding flat Yukawa tensor is obtained from \eqref{eq:Yukawas} by multiplying by \(\Pi_0^{-2}\) and applying the Jacobian of \((t_H,t_F)\mapsto(\log z_1,\log z_2)\) to all 3 indices.

Define period-derived coefficients \(n^{\rm per}_{n,m}\) by
\[
 \mathcal F_{\mathrm{inst}}=\sum_{(n,m)\in\mathbb Z_{\ge0}^2\setminus\{(0,0)\}}n^{\rm per}_{n,m}\Li_3(Q_H^nQ_F^m),
\]
where \(\Li_3(q)=\sum_{k\ge1}q^k/k^3\) is the trilogarithm.
We regard these as B-model predictions for the genus-zero BPS numbers of $X_-$.
The values needed for the independent checks below are
\[
n^{\rm per}_{0,1}=16,\qquad n^{\rm per}_{1,0}=176,\qquad n^{\rm per}_{2,0}=100.
\]
The first number equals the number of flopping curves in the primitive fiber class; the latter two will be recovered by direct intersection theory on $X_-$.
More generally, multiple covers are removed from the formal prepotential by
\[
n^{\rm per}_{n,m}=[Q_H^nQ_F^m]\mathcal F_{\mathrm{inst}}-
 \sum_{\substack{k\mid n,\ k\mid m\\k>1}}
                  \frac{n^{\rm per}_{n/k,m/k}}{k^3}.
\]

\begin{proposition}\label{prop:firstrow}
As a formal series in \(Q_H\) with coefficients in \(\QQ[[Q_F]]\),
\begin{equation}
 \mathcal F_{\mathrm{inst}}=16\Li_3(Q_F)+Q_H(176+832Q_F+426Q_F^2)+O(Q_H^2).
 \label{eq:firstrow}
\end{equation}
\end{proposition}
\begin{proof}
On the line \(z_1=0\) the series \(\widehat{\mathcal I}_-\) reduces to
\(\sum_{m\ge0} r_{0,m}(\Hm,\Fm)z_2^m\), which has no part of degree 1, since
\(r_{0,m}\) is divisible by \((2x-y)^2\) for \(m\ge1\); hence \(\Pi_0=1\)
and \(m_H=m_F=0\) there, and \eqref{eq:YukawaBoundary} integrates to
\(16\Li_3(Q_F)\) with zero constant correction. At order \(z_1\) write
\[
m_H^{(1)}=8+144z_2+106z_2^2,\quad m_F^{(1)}=16+16z_2-32z_2^2,\quad
 \phi^{(1)}=4+32z_2+12z_2^2.
\]
These polynomials are the exact \(z_1^1\)-coefficients of \(m_H,m_F,\Pi_0\); the tail
\(m>2n\) starts in cohomological degree 2 and does not affect them.
The \(z_1^1\)-coefficient of \(C_{111}\) is
\(784+4688z_2+2052z_2^2\). Normalizing the form and transforming all 3
indices gives, for the flat Yukawa coupling \(K_{HHH}=\partial_{t_H}^3\mathcal F\),
\[
[Q_H]K_{HHH}=784+4688z_2+2052z_2^2-21m_H^{(1)}-24m_F^{(1)}-14\phi^{(1)} =176+832z_2+426z_2^2
\]
(the coefficients are \(3\cdot7\) from the Jacobian on the 3 indices,
\(3\cdot8\) from the cross terms with \(C_{112}\), and \(2\cdot7\) from
\(\Pi_0^{-2}\)).
To this order \(z_2\) can be replaced by \(Q_F\).
Since 3
\(t_H\)-derivatives do not change a term of \(Q_H\)-degree 1, this proves
the formula.
\end{proof}

The coefficient \(16\) of the primitive fiber class agrees with the 16 flopping curves.

\subsection{Independent A-model tests on the nef chamber}
\label{sec:independentcounts}
We now give calculations that do not use the preceding period coefficients.
On \(X_-\), degrees mean \((n,m)=(\Hm\cdot C,\Fm\cdot C)\).
All genericity statements in this section refer to general defining sections.

\begin{proposition}\label{prop:176lines}
There are \(176\) vertical lines on \(X_-\to\PP^2\).
They are isolated
\((-1,-1)\)-curves, and their degree is \((1,0)\).
\end{proposition}
\begin{proof}
Parametrize lines by \(\mathcal G=\PP^2\times\operatorname{Gr}(2,5)\), and put
\(\gamma=c_1\OO_{\PP^2}(1)\). Let \(\mathcal S\) be the rank-2
tautological bundle on the Grassmannian.
Restricting the two quadrics and the linear equation to the universal line gives the rank-8 bundle
\[
 \mathcal B=(\mathcal Q\boxtimes\operatorname{Sym}^2\mathcal S^*)
       \oplus(\OO(1)\boxtimes\mathcal S^*),
\]
where \(\mathcal Q\) is the bundle \eqref{eq:Qbundle}, so that
\(c(\mathcal Q)=1+2\gamma+3\gamma^2\).
The defining section space generates \(\mathcal B\); hence a general zero locus is transverse and zero dimensional.
If \(\alpha,\beta\) are the Chern roots of \(\mathcal S^*\), the roots of \(\operatorname{Sym}^2\mathcal S^*\) are
\(2\alpha,\alpha+\beta,2\beta\).
Hence
\[
c_8(\mathcal B)=(\alpha+\gamma)(\beta+\gamma)
 \prod_{r\in\{2\alpha,\alpha+\beta,2\beta\}}(r^2+2\gamma r+3\gamma^2).
\]
Since \(\int_{\PP^2}\gamma^2=1\), only the \(\gamma^2\)-coefficient matters:
\[
p(\alpha,\beta):=[\gamma^2]c_8(\mathcal B)
=4\alpha\beta(7\alpha^4+46\alpha^3\beta+90\alpha^2\beta^2+46\alpha\beta^3+7\beta^4).
\]
On \(\operatorname{Gr}(2,5)\) the integral of a symmetric polynomial \(p\)
of degree 6 in the Chern roots is
\([\alpha^4\beta^3](\alpha-\beta)p\), since the top Schubert class is \(\alpha^3\beta^3\).
Here
\(p=28\alpha^5\beta+184\alpha^4\beta^2+360\alpha^3\beta^3
+184\alpha^2\beta^4+28\alpha\beta^5\), so the integral is \(360-184=176\).
Thus \(\int_{\mathcal G}c_8(\mathcal B)=176\).
Transversality implies \(H^0(N_{C/X_-})=0\), and adjunction on the Calabi--Yau threefold \(X_-\) gives \(\deg N=-2\).
Thus
\(N=\OO(-1)^{\oplus2}\), and every line contributes 1.
\end{proof}

\begin{proposition}\label{prop:100conics}
There are \(100\) smooth vertical conics on \(X_-\), each an isolated
\((-1,-1)\)-curve, and the degree-\((2,0)\) genus-zero BPS coefficient is \(100\).
\end{proposition}
\begin{proof}
The varying hyperplane defines a rank-4 bundle
\[
0\longrightarrow K\longrightarrow\OO_{\PP^2}^{\oplus5}
                       \longrightarrow\OO_{\PP^2}(1)\longrightarrow0.
\]
The two quadrics on its fibers form a pencil.
On
\(\PP(\mathcal Q^*)\) in the convention of lines, put
\(\tau=c_1\OO(1)\). The relation and integrals are
\[
 \tau^2-2\gamma\tau+3\gamma^2=0,\quad
 \int \gamma^2\tau=1,\quad \int \gamma\tau^2=2,\quad \int\tau^3=1.
\]
Contracting the pencil with the tautological line gives a symmetric form on \(K\) with values in \(\OO(1)\).
The corank-2 locus has codimension 3.
The symmetric degeneracy formula \cite{HTu} gives its class
\[
4(c_1c_2-c_3)\bigl(K^*\otimes\OO(\tau/2)\bigr) =10\gamma^2\tau+15\gamma\tau^2+10\tau^3.
\]
(The half twist is formal Chern-root notation.) The integral is \(10+30+10=50\).
The section space is fiberwise generating, so a general degeneracy locus is reduced; the higher-corank stratum is avoided by dimension.
Each rank-2 quadric is a union of two planes, and the other quadric cuts a smooth conic on each; hence the 50 rank-2 members give
\(2\cdot50=100\) conics. The additional conditions for a singular
conic, or two such pencil members over the same base point, have higher codimension and are avoided for general sections.
Conversely every vertical conic spans a plane and produces such a rank-2 pencil member.
The determinantal incidence is transverse, so the resulting conics are isolated reduced points of the Hilbert scheme.
Adjunction again gives the normal bundle \(\OO(-1)^{\oplus2}\).

Two intersecting lines in a fiber would give an additional, higher-codimension pencil degeneracy; thus no such pair occurs for general sections, and the conics are the only primitive curves of degree \((2,0)\).
\end{proof}

\section{Concluding remarks}\label{sec:conclusion}

The example studied here shows how the syzygies of a negative column lead to an auxiliary VGIT presentation in which another chamber gives an ordinary nef complete intersection.
The generalized phase $X_+$ and the ordinary phase $X_-$ are two adjacent chambers of the same birational geometry, related by the flop of 16 disjoint $(-1,-1)$-curves.
On the B-model side, a single compact mirror family has two maximally unipotent boundary points, $w=0$ and $z=0$, whose indicial algebras are isomorphic to the rational even cohomology rings of $X_+$ and $X_-$, respectively.
The Yukawa couplings and the first genus-zero predictions provide further checks of these two boundary interpretations.

The ordinary chamber is especially useful because it converts the negative gCICY presentation into a nef complete intersection in a toric ambient space.
When such a chamber is Gorenstein, the Batyrev--Borisov construction is the natural next step.
The present ambient space is non-Gorenstein, so a different route is required: the Hori--Vafa equations select the Laurent model, 
while the compactification, crepant resolution, Hodge numbers, period system, and genus-zero checks are established independently.

The role of the Gorenstein condition for $[3,-1]$~blocks is treated in
\cite{KanII}, and gCICYs with a $[4,-2]$~block in \cite{KanIII}.

\appendix

The appendices collect calculations that support the geometric and differential-system statements in the main text but are not needed for its logical flow.
Appendix~A verifies the Mordell--Weil group.
Appendix~B records the local and global data used in the crepant resolution.
Appendix~C analyzes the $A_2$ cusp of the parameter discriminant.
Appendix~D gives algebraic certificates for the Picard--Fuchs system and the two maximally unipotent boundary charts.
Appendix~E verifies the horizontal pairing and the Yukawa couplings.

\section{Arithmetic verification of the Mordell--Weil group}\label{app:MW}
This appendix supplies the arithmetic specialization used in Proposition~\ref{prop:MW}.
\begin{proof}[Proof of Proposition~\ref{prop:MW}]
Put $k=\CC(\lambda)$ and express $\CC(a,b)$ as the function field of the smooth projective generic fiber $C_\lambda$ of $\lambda=T/S$.
Set $r=(1+a)(1+b)/(ab)$ and $s=ab$.
Then
\[
s=\frac{1-(z_1/\lambda)r^2}{z_2},\qquad a+b=(r-1)s-1.
\]
Consequently $C_\lambda$ is the genus-2 curve
\begin{equation}
 \zeta^2=\{(r-1)(1-(z_1/\lambda)r^2)-z_2\}^2
                   -4z_2(1-(z_1/\lambda)r^2).\label{eq:genus2}
\end{equation}
It is geometrically integral and smooth generically.
We claim that its Jacobian has no nonzero geometric homomorphism to $E_\lambda$ generically.
Specialize $z_1=\lambda=1,z_2=2$.
Its sextic is
\[
P(r)=r^6-2r^5-r^4+8r^3+3r^2-6r+1, \qquad \operatorname{disc}(P)=2^{21}\cdot73.
\]
Both this curve and $E_1$ have good reduction at $7$.
Direct point counting and a Hasse--Witt computation give
\[
 \#C(\mathbb F_7)=9,\quad \#C(\mathbb F_{49})=33,\quad
 H_C=\begin{pmatrix}4&1\\2&2\end{pmatrix}\pmod7.
\]
Thus $\det H_C=6\ne0\pmod7$, so its Jacobian is ordinary.
Equivalently its Weil polynomial is
\[
T^4+T^3-8T^2+7T+49.
\]
On the other hand $E_1:y^2=x^3-7x^2+16x$ has 8 points over $\mathbb F_7$ and is supersingular.
An ordinary abelian variety has no nonzero homomorphism to a supersingular elliptic curve, even after an algebraic extension: their rational Dieudonn\'e modules have disjoint Newton slopes.
Hence
\[
\operatorname{Hom}\big(\operatorname{Jac}(C\otimes\overline{\mathbb F}_7), E_1\otimes\overline{\mathbb F}_7\big)=0 .
\]
Choose a connected normal integral model of the smooth family containing this good reduction.
Specialization of homomorphisms between abelian schemes is injective, so the vanishing over $\overline{\mathbb F}_7$ implies
\[
 \operatorname{Hom}\bigl(\operatorname{Jac}(C_\lambda),E_\lambda\bigr)=0
\]
at the geometric generic point of the family over the $(z_1,z_2,\lambda)$-space.
The jumping locus for this Hom group is a countable union of proper closed subsets of that space.
Hence for very general complex $(z_1,z_2)$ it does not contain the whole $\lambda$-line, and the vanishing holds over $\overline{\CC(\lambda)}$.

An $E_\lambda$-point over $k(C_\lambda)$ gives a rational map $C_\lambda\dashrightarrow E_\lambda$, hence a morphism.
Since the indicated geometric Hom group vanishes, this morphism is geometrically constant.
As $C_\lambda$ is geometrically integral, the constant descends to $k$.
There are therefore no new points after extension from $k$ to $k(C_\lambda)$.
The standard rational elliptic surface calculation after
\eqref{eq:universal4} gives $E_\lambda(k)=\mathbb Z/4\mathbb Z$, proving the
claim.
\end{proof}

\section{Crepant resolution data}\label{app:resolutiondata}
\subsection{Non-minimal points of the Tate model}\label{app:nonminimal}

\begin{lemma}\label{lem:nonminimal}
Let $p$ be a point of a smooth surface and consider the Tate form
\eqref{eq:Tate4} with local sections $\alpha,\beta$.  The short Weierstrass
coefficients $f,g$ vanish at $p$ to orders at least $(4,6)$ if and only if $\operatorname{ord}_p\alpha\ge1$ and $\operatorname{ord}_p\beta\ge2$.
For $(z_1,z_2)$ with $z_1z_2(z_2-1)\ne0$, the points of $\mathcal B_0=\PP^1_a\times\PP^1_b$ with this property are exactly $p_1,\ldots,p_6$ of Section~\ref{sec:crepant}.
\end{lemma}

\begin{proof}
Recall $f=-c_4/48$, $g=-c_6/864$ with $c_4=b_2^2-24b_4$ and $c_6=-b_2^3+36b_2b_4-216b_6$, where $b_2=\alpha^2+4\beta$, $b_4=\alpha^2\beta$ and $b_6=\alpha^2\beta^2$.
If $\operatorname{ord}\alpha\ge1$ and $\operatorname{ord}\beta\ge2$, then $\operatorname{ord}b_2\ge2$, $\operatorname{ord}b_4\ge4$ and $\operatorname{ord}b_6\ge6$, so $\operatorname{ord}c_4\ge4$ and $\operatorname{ord}c_6\ge6$.
Conversely, if $\alpha(p)\ne0$, put $t=\beta/\alpha^2$; then
\[
c_4=\alpha^4(1-16t+16t^2),\qquad c_6=-\alpha^6(1-24t+120t^2+64t^3),
\]
and the two polynomials in $t$ have no common root (their resultant is $-27648\ne0$), so $f$ and $g$ do not both vanish at $p$.
If $\alpha(p)=0$ and $\beta(p)\ne0$, then $c_4(p)=16\beta(p)^2\ne0$.
If $\alpha(p)=0$ and $\operatorname{ord}_p\beta=1$, then $c_4=\alpha^4-16\alpha^2\beta+16\beta^2$ has order exactly 2.
This proves the first statement.

For the second, $\alpha=2a_1b_1\mathsf L$ vanishes on $\{a=0\}\cup\{b=0\}\cup\{\mathsf L=0\}$, 
and $\beta=4\mathsf LT$ with $T=z_1a_0b_0(a_0+a_1)^2(b_0+b_1)^2$ vanishes on $\{\mathsf L=0\}$, $\{a=\infty\}$, $\{b=\infty\}$ with multiplicity 1 and on $\{a=-1\}$, $\{b=-1\}$ with multiplicity 2.
Hence $\operatorname{ord}_p\beta\ge2$ exactly when $p$ lies on $\{a=-1\}$ or $\{b=-1\}$, or on two of the three curves $\{\mathsf L=0\}$, $\{a=\infty\}$, $\{b=\infty\}$.
\begin{itemize}
\item On $\{a=-1\}$, $\alpha=0$ requires $b=0$, giving $p_4=(-1,0)$, or
$\mathsf L=1+z_2b=0$, giving $p_5=(-1,-1/z_2)$.
\item Symmetrically, on $\{b=-1\}$ one obtains $p_3=(0,-1)$ and
$p_6=(-1/z_2,-1)$.
\item $\{a=\infty\}\cap\{b=\infty\}$ is the point $(\infty,\infty)$, where
$\mathsf L=-z_2a_1b_1\ne0$ and $\alpha\ne0$.
\item $\{\mathsf L=0\}\cap\{a=\infty\}$ forces $b=0$, giving $p_2$, and
$\{\mathsf L=0\}\cap\{b=\infty\}$ gives $p_1$; there $\alpha=0$ because $\mathsf L=0$.
\end{itemize}
Finally, at $(-1,-1)$, which lies on both $\{a=-1\}$ and $\{b=-1\}$, one has $\mathsf L=1-z_2\ne0$ and $a_1b_1\ne0$, so $\alpha\ne0$; this is where the condition $z_2\ne1$ is used.
The 6 points are distinct because $z_2\notin\{0,1\}$.
\end{proof}

\subsection{Intersection numbers on \texorpdfstring{$\mathcal B_8$}{B8}}\label{app:intersections}
With $\fa^2=\fb^2=0$, $\fa\fb=1$ and $e_ie_j=-\delta_{ij}$, the classes of the discriminant components in Section~\ref{sec:crepant} have the following intersection numbers.
(Here $B_0$ stands for $B_0'$ and $\mathsf L$ for $\mathsf L'$.)
\begin{center}
\small
\setlength{\tabcolsep}{4pt}
\begin{tabular}{c|rrrrrrrrrrrr}
 & $A_0$ & $B_0$ & $E_1$ & $E_2$ & $A_\infty$ & $B_\infty$ & $E_5$ & $E_6$ & $A_{-1}$ & $B_{-1}$ & $\mathsf L$ & $\Dres$\\ \hline
$A_0$ & $-2$ & 1 & 1 & 0 & 0 & 0 & 0 & 0 & 0 & 0 & 0 & 0\\
$B_0$ & 1 & $-2$ & 0 & 1 & 0 & 0 & 0 & 0 & 0 & 0 & 0 & 0\\
$E_1$ & 1 & 0 & $-1$ & 0 & 0 & 1 & 0 & 0 & 0 & 0 & 1 & 1\\
$E_2$ & 0 & 1 & 0 & $-1$ & 1 & 0 & 0 & 0 & 0 & 0 & 1 & 1\\
$A_\infty$ & 0 & 0 & 0 & 1 & $-1$ & 1 & 0 & 0 & 0 & 1 & 0 & 2\\
$B_\infty$ & 0 & 0 & 1 & 0 & 1 & $-1$ & 0 & 0 & 1 & 0 & 0 & 2\\
$E_5$ & 0 & 0 & 0 & 0 & 0 & 0 & $-2$ & 0 & 1 & 0 & 0 & 0\\
$E_6$ & 0 & 0 & 0 & 0 & 0 & 0 & 0 & $-2$ & 0 & 1 & 0 & 0\\
$A_{-1}$ & 0 & 0 & 0 & 0 & 0 & 1 & 1 & 0 & $-2$ & 1 & 0 & 0\\
$B_{-1}$ & 0 & 0 & 0 & 0 & 1 & 0 & 0 & 1 & 1 & $-2$ & 0 & 0\\
$\mathsf L$ & 0 & 0 & 1 & 1 & 0 & 0 & 0 & 0 & 0 & 0 & $-4$ & 0\\
$\Dres$ & 0 & 0 & 1 & 1 & 2 & 2 & 0 & 0 & 0 & 0 & 0 & 4
\end{tabular}
\end{center}
The positive off-diagonal entries are realized as follows: the 9 ordinary multiplicative crossings of Section~\ref{sec:crepant}; $E_1\cdot B_\infty=E_1\cdot\Dres=1$ and $B_\infty\cdot\Dres=2$ at the triple point $q_1$, and $E_2\cdot A_\infty=E_2\cdot\Dres=1$ and $A_\infty\cdot\Dres=2$ at $q_2$, the entries 2 being the tangencies; and $E_i\cdot\mathsf L'=1$ at $r_i$.
Since the local intersection multiplicities at these points already exhaust the intersection numbers, there are no further intersections.
The classes in Section~\ref{sec:crepant} also give the adjunction data $K\cdot\mathsf L'=2$, $K\cdot\Dres=-2$, so that $\mathsf L'$ is rational and $\Dres$ has genus 2, 
and the Euler numbers of the open curves used in Proposition~\ref{prop:euler88}: each component has Euler number $2-2g$ minus the number of collision points on it.

\subsection{The residual discriminant curve}\label{app:residualcurve}
This appendix supplies the local calculations used in Proposition~\ref{prop:residualgenus}.
The original residual curve has bidegree $(3,3)$ and arithmetic genus 4.
At $p_3=(0,-1)$, using $a=u,b=-1+v$, its tangent cone is $u^2-16z_1v^2$; the analogous cone at $p_4$ is the same.
Thus these are 2 ordinary nodes when $z_1\ne0$, each resolved by its prescribed blowup.
The residual curve is smooth at $p_1,p_2,p_5,p_6$ and at their relevant infinitely near points.
Away from these points, generic smoothness follows from the pencil.
Generic irreducibility follows from the genus-2 function-field presentation \eqref{eq:genus2}, specialized to $\lambda=1/16$.
For example $z_1=1/16,z_2=2$ gives the square-free sextic $P(r)$ in Appendix~\ref{app:MW}, so this specialization is connected and has genus 2.

The class in the discriminant table gives $\Dres^2=4$ and $K_{\mathcal B_8}\Dres=-2$, hence adjunction gives $g(\Dres)=2$.
For the tangency at $p_1$, use $a=u,b=1/v$.
The residual equation is
\[
u^2(v-z_2u)-16z_1v(1+u)^2(1+v)^2=0.
\]
After $v=ut$ and division by $u$, its restriction to $u=0$ is $-16z_1t$, while its restriction to $t=0$ is $-z_2u^2$.
Hence the strict transform is smooth, transverse to $E_1=\{u=0\}$, and tangent of order 2 to $B_\infty=\{t=0\}$.
The other incidence is symmetric.
The divisor classes exclude additional intersections with the remaining discriminant components.

\subsection{Local resolutions at the collision points}\label{app:local}
This appendix contains the chart-by-chart computations behind
\eqref{eq:eA} and \eqref{eq:eB}, which were only summarized in
Section~\ref{sec:crepant}.
Throughout, $u,v$ are local coordinates on $\mathcal B_8$ and $x,y$ are the affine Weierstrass coordinates of \eqref{eq:Tate4}.

\subsubsection{\texorpdfstring{The triple points $q_1,q_2$}{The triple points q1, q2}}\label{app:triple}
At $q_1,q_2$ the transformed Tate coefficients are transverse local coordinates, so we may take $\alpha=u$, $\beta=v$.
The local equation is
\[
F_A=y^2+uxy+uvy-x^3-vx^2=0,
\]
with discriminant $u^2v^4(u^2-16v)$.
Blow up the smooth singular curve $(x,y,v)=0$.
In the $v$-chart, $x=vX$, $y=vY$, the strict transform is
\begin{equation}
 Y^2+u(X+1)Y-vX^2(X+1)=0.\label{eq:Afirst}
\end{equation}
The two remaining singular curves are $(X+1,Y,v)=0$ and $(X,Y,u)=0$; they are disjoint.
The first is an ordinary $A_1$ family.
Blowing up the second, the $X$-chart is
\begin{equation}
 y_1^2+u_1(1+X)y_1-v(1+X)=0,\label{eq:Alast}
\end{equation}
which is smooth near the exceptional locus; the other charts are smooth by the same unit-derivative check.
Thus 3 curve blowups give a smooth crepant local model.

For the Euler number it is enough to record the reduced exceptional fibers.
Starting from the cuspidal cubic of Euler number 2, the 3 blowups add fibers of Euler numbers $2,2,3$.
Hence
\begin{equation}
 e(F_{q_i})=2+(2-1)+(2-1)+(3-1)=6.\label{eq:eAapp}
\end{equation}
The resolved fiber has 5 reduced rational components.
This differs from the Kodaira type seen on a transverse one-parameter slice, because resolving the slice and restricting the resolved threefold are not the same operation.

\subsubsection{\texorpdfstring{The additive--multiplicative crossings $r_1,r_2$}{The additive--multiplicative crossings r1, r2}}\label{app:addmult}
At $r_1,r_2$ we may choose coordinates with $\beta=v$, $\alpha=uv$, so
\[
F_B=y^2+uvxy+uv^2y-x^3-vx^2=0.
\]
After blowing up $(x,y,v)$, the $v$-chart is
\begin{equation}
 Y^2+uv(X+1)Y-vX^2(X+1)=0.\label{eq:Bfirst}
\end{equation}
The curve $(X+1,Y,v)=0$ is an ordinary $A_1$ family.
Resolving it and then the residual curve $(X,Y,v)=0$ gives the two representative charts
\begin{align}
 y_1^2+u(1+vx_1)y_1-vx_1^2(1+vx_1)&=0,\label{eq:Bsecondv}\\
 y_1^2+uv_1(1+X)y_1-Xv_1(1+X)&=0.\label{eq:Bsecondx}
\end{align}
The second contains one further ordinary $A_1$ curve, while the first contains the remaining curve $(x_1,y_1,u)=0$.
One blowup of each resolves them; in the resulting charts a linear derivative is a unit along the central fiber.
Hence 5 curve blowups give a smooth crepant local model.
Their reduced exceptional fibers have Euler numbers $2,2,2,2,3$, and therefore
\begin{equation}
 e(F_{r_i})=2+4(2-1)+(3-1)=8.\label{eq:eBapp}
\end{equation}
There are 7 reduced rational components.
Away from $u=0$ the same construction restricts to the usual $I_1^*$ resolution.

\subsection{Global resolution data}\label{app:globalresolution}
This appendix records the global ordering and chart coverage used in Proposition~\ref{prop:globalresolution}.
It complements the local collision calculations of Appendix~\ref{app:local}.
\begin{proof}[Proof of Proposition~\ref{prop:globalresolution}]
Start with the smooth projective bundle ambient of the projective Weierstrass equation.
All singularities lie in its affine $z=1$ chart; the zero section at infinity is smooth and is not altered.
The global ordering of the $24$ curve blowups is summarized in Table~\ref{tab:blowups}; we then explain why each center is globally well defined and why the local computations of Appendix~\ref{app:local} and of the multiplicative crossings \eqref{eq:evenCA} in Section~\ref{sec:crepant} cover every chart.

\begin{table}[ht]
\centering\small
\begin{tabular}{@{}clcl>{\raggedright\arraybackslash}p{5.2cm}@{}}
\toprule
Stage & Base curve(s) & Fiber & Blowups & Centers (successive singular sections) \\
\midrule
1 & $\mathsf L'$ & $I_1^*$ & $4$ & original singular section; $A_1$ section ``$X=-1$''; section ``$X=0$''; $A_1$ section of \eqref{eq:Bsecondx}\\
2 & $A_\infty,B_\infty,E_5,E_6$ & $I_4$ & $4\times2$ & singular section, then its residual singular section\\
3 & $A_{-1},B_{-1}$ & $I_8$ & $2\times4$ & successive singular sections (4 steps)\\
4 & $A_0,B_0',E_1,E_2$ & $I_2$ & $4\times1$ & the unique singular section\\
\midrule
 & & & $24$ & \\
\bottomrule
\end{tabular}
\caption{Global order of the curve blowups resolving $\mathcal W_8$.  Within each stage the base curves are treated in the order listed.}\label{tab:blowups}
\end{table}

The 4 centers over $\mathsf L'$ are the successive singular sections appearing in the local model of Appendix~\ref{app:addmult}: the original singular section, 
the $A_1$ section distinguished by $X=-1$, the section $X=0$, and the $A_1$ section of \eqref{eq:Bsecondx}.
These sections are globally distinguished by the split Tate data.
Next resolve the split multiplicative components: 2 blowups for each $I_4$ curve and 4 for each $I_8$ curve.
Finally resolve the 4 $I_2$ components by their unique residual singular sections.
At $q_i$ and $r_i$ these are exactly the local centers described in Appendices~\ref{app:triple} and \ref{app:addmult}; at ordinary multiplicative crossings the even-exponent model
\eqref{eq:evenCA} shows that the closures are smooth and that the unused
charts are nonsingular.
Thus all $24$ centers in Table~\ref{tab:blowups} are globally defined smooth curves.

Each center has ambient codimension 3 and the Weierstrass hypersurface has multiplicity 2 along it, so every blowup has discrepancy $3-1-2=0$.
The three local collision models of Section~\ref{sec:crepant} exhaust the singular locus; away from the collision points one has the standard transversal ADE resolutions.
Hence the resulting total space is smooth, projective, and crepant.

Finally, every exceptional fiber over a base point is a nonzero conic in $\PP^2$ (possibly reducible or nonreduced), never the whole $\PP^2$.
Therefore no surface lies over a point of $\mathcal B_8$, and the resolved elliptic morphism is equidimensional, hence flat.
Triviality of the canonical bundle follows from crepancy and the anticanonical fundamental line bundle.
Since $\pi_*\OO_Y=\OO_{\mathcal B_8}$ and $R^1\pi_*\OO_Y=K_{\mathcal B_8}$, Leray and rationality of $\mathcal B_8$ give the asserted vanishings.
\end{proof}

\section{\texorpdfstring{Parameter discriminant and the $A_2$ cusp}{Parameter discriminant and the A2 cusp}}\label{app:parameterdiscriminant}\label{app:discriminantcalc}
We give the local calculation behind Proposition~\ref{prop:discriminant}(3), writing $x_1,\ldots,x_4$ for $a,b,c,d$.
By Proposition~\ref{prop:discriminant}(1), the singular points are $a=b=t$, $c=d=1$ with parameters \eqref{eq:param}.
At such a point the logarithmic Hessian has determinant
\[
\det\big((x_i\partial_{x_i})(x_j\partial_{x_j})G\big)
=\frac{3t(t+2)}{(t+3)^4}.
\]
Thus it is nondegenerate generically.
At $t=-2$ it has rank 3.
Along its kernel, set $a=b=-2e^s$, $c=d=1$; then
\[
G=1-2e^{2s}+(e^{-s}-2)^4,\qquad G^{(3)}(0)=-24.
\]
The holomorphic splitting lemma gives the local equation $x_1^2+x_2^2+x_3^2+x_4^3=0$.
The constant and linear jets of $G_{z_1},G_{z_2}$ along this kernel form the matrix
$\left(\begin{smallmatrix}-1&-4\\-4&-8\end{smallmatrix}\right)$, whose
determinant is $-8$, so the two parameters give a miniversal deformation.
Finally, the determinant of the second and third derivatives of \eqref{eq:param} at $t=-2$ is $9/2$, confirming an ordinary cusp of the parameter discriminant.

\section{Differential-system calculations}\label{app:pfcalc}

This appendix contains the verifications used in Section~\ref{sec:periods}.
We write $\theta_i=\theta_{z_i}$ and
\begin{align*}
 P_0(a,b)&=4a^2-8ab+9b^2,\\
 P_1(a,b)&=4a^2+4ab+3b^2+4a+2b+1,\\
 P_2(a,b)&=8a^2+5ab-3b^2+9a-2b+1,
\end{align*}
so that $\mathcal L_2=P_0(\theta)-9z_2(2\theta_1-\theta_2)^2-16z_1P_1(\theta)-16z_1z_2P_2(\theta)$.

\subsection{Recurrences and boundary rows}\label{app:pfrec}
Let $x,y$ be complex variables and put $s=n+x$, $t=m+y$.
The coefficient of $z_1^{\Hm}z_2^{\Fm}z_1^nz_2^m$ in $\mathcal L_2\mathcal I_-$ is
\begin{multline}\label{eq:L2row}
 P_0(s,t)\,r_{n,m}-9(2s-t+1)^2\,r_{n,m-1}\\
 -16P_1(s-1,t)\,r_{n-1,m}-16P_2(s-1,t-1)\,r_{n-1,m-1},
\end{multline}
evaluated at $(x,y)=(\Hm,\Fm)$, where terms with a negative index are absent.
The Gamma function gives the meromorphic identities
\begin{align*}
\frac{C(s,t-1)}{C(s,t)}&=\frac{t^3}{(s+t)(2s-t+1)^2},\\
\frac{C(s-1,t)}{C(s,t)}&=\frac{s(2s-t)^2(2s-t-1)^2}{16(2s-1)^4(s+t)},\\
\frac{C(s-1,t-1)}{C(s,t)}&=\frac{s\,t^3(2s-t)^2}{16(2s-1)^4(s+t-1)(s+t)}.
\end{align*}
Hence \eqref{eq:L2row} equals $r_{n,m}(x,y)\,E(s,t)$, where
\begin{multline*}
 E(s,t)=P_0(s,t)-\frac{9t^3}{s+t}
 -\frac{s(2s-t)^2(2s-t-1)^2P_1(s-1,t)}{(2s-1)^4(s+t)}\\
 -\frac{s\,t^3(2s-t)^2P_2(s-1,t-1)}{(2s-1)^4(s+t-1)(s+t)}.
\end{multline*}
The rational function $E$ vanishes identically.
After multiplication by $(2s-1)^4(s+t)(s+t-1)$ it becomes a polynomial of degree at most 8 in each variable, and the identity can be checked by expansion or by evaluation on a $9\times9$ grid.
Thus for $n,m\ge1$ the expression \eqref{eq:L2row} is a holomorphic function of $(x,y)$ near $0$ that vanishes identically, and so do its Taylor coefficients.
For $\mathcal L_3$ the corresponding statement is the first identity above: $t^3r_{n,m}-(s+t)(2s-t+1)^2r_{n,m-1}=0$.

\paragraph{Boundary rows.}\label{app:pfboundary}
Put $u=2\Hm-\Fm$, $A=u^2$ and $B=-8\Hm\Fm^2$.
The relations
\[
4\Hm^2-8\Hm\Fm+9\Fm^2=0,\qquad \Fm^3=0
\]
in $\Rminus$ give
\[
A=(2\Hm-\Fm)^2=4\Hm\Fm-8\Fm^2, \qquad
\Fm A=4\Hm\Fm^2.
\]
Multiplying the quadratic relation by $\Fm$ also gives
\[
\Hm A=4\Hm^2\Fm-8\Hm\Fm^2=0.
\]
Therefore
\[
uA=(2\Hm-\Fm)A=-\Fm A, \qquad
B=-8\Hm\Fm^2=-2\Fm A=2uA.
\]
Directly from the coefficient formula, for $m\ge1$,
\[
r_{0,m}(\Hm,\Fm)=\frac{A}{m^2}+\frac{B}{m^3}.
\]

\emph{Row $(0,0)$.}  The two equations read $P_0(\Hm,\Fm)=0$ and
$\Fm^3=0$, which are the relations of $\Rminus$.

\emph{Rows $(n,0)$ with $n\ge1$.}  The absent terms of \eqref{eq:L2row} are
those with $r_{n,-1}$ and $r_{n-1,-1}$.
Define $r_{n,-1}(x,y)=C(n+x,y-1)/C(x,y)$.
It contains the factor $\Gamma(1+y)^3/\Gamma(y)^3=y^3$, and its other factors are holomorphic near $0$; the same holds for $r_{n-1,-1}$.
The identity of Appendix~\ref{app:pfrec} holds for the complete 4-term expression.
Since the absent terms are multiples of $\Fm^3=0$, the 3-term expression that actually occurs also vanishes.
The same argument applies to $\mathcal L_3$.

\emph{Rows $(0,m)$ with $m\ge1$.}  Here the absent terms involve $r_{-1,\cdot}$, which is not divisible by a relation of $\Rminus$, so we check these rows directly.
The $\mathcal L_2$-row reads $P_0(\Hm,m+\Fm)r_{0,m}=9(u-m+1)^2r_{0,m-1}$.
Since $P_0(\Hm,m+\Fm)=9m^2+m(18\Fm-8\Hm)+P_0(\Hm,\Fm)$ and $P_0(\Hm,\Fm)=0$,
\[
P_0(\Hm,m+\Fm)\,r_{0,m}=9A+\frac{9B+(18\Fm-8\Hm)A}{m}=9A .
\]
For $m=1$ the right side is $9u^2=9A$.
For $m\ge2$ it is
\[
9\bigl((m-1)^2-2(m-1)u\bigr)\Bigl(\frac{A}{(m-1)^2}+\frac{B}{(m-1)^3}\Bigr) =9A+\frac{9(B-2uA)}{m-1}=9A .
\]
The $\mathcal L_3$-row reads $(m+\Fm)^3r_{0,m}=(\Hm+\Fm+m)(u-m+1)^2r_{0,m-1}$.
The left side is $mA+B+3\Fm A=mA+\Fm A$.
For $m=1$ the right side is $(1+\Hm+\Fm)A=A+\Fm A$.
For $m\ge2$ it is
\[
(m+\Hm+\Fm)\Bigl(A+\frac{B-2uA}{m-1}\Bigr)=mA+\Fm A .
\]
This completes the proof of Proposition~\ref{prop:minuschart}(1), and, taking degree-zero parts, of Proposition~\ref{prop:PF}.

\subsection{Principal symbols and logarithmic frames}\label{app:pfsymbol}
The principal symbols are
\begin{align*}
\sigma(\mathcal L_2)={}&(4-64z_1-36z_2-128z_1z_2)\xi_1^2
 +(-8-64z_1+36z_2-80z_1z_2)\xi_1\xi_2\\
 &+(9-48z_1-9z_2+48z_1z_2)\xi_2^2,\\
\sigma(\mathcal L_3)={}&\xi_2^3-z_2(\xi_1+\xi_2)(2\xi_1-\xi_2)^2 .
\end{align*}
Their resultant, computed as the $5\times5$ Sylvester determinant, equals $-64(z_2-1)\Delta$.

For the birationality of \eqref{eq:param}, take $t=1/2$, so that $(z_1,z_2)=(1/3024,16/7)$.
The other two solutions of $t^2(t+3)=7/8$ are $t=(-7\pm\sqrt{21})/4$, and they give irrational values of $z_1$.
Hence a general fiber of $t\mapsto(z_1,z_2)$ is a single point.

\paragraph{Logarithmic frames.}\label{app:pfframe}
Near $z=0$, consider the 15 monomials $\theta^\alpha$ with $|\alpha|\le4$ and the 9 elements
\[
 \theta^\alpha\mathcal L_2\ (|\alpha|\le2),\qquad \theta^\alpha\mathcal L_3\ (|\alpha|\le1)
\]
of $\mathfrak L_-$, written in the basis of monomials after moving coefficients to the left.
Let $R_9(z)$ be the $9\times9$ minor of the coefficient matrix on the 9 monomials not in the frame
\[
1,\ \theta_1,\ \theta_2,\ \theta_1\theta_2,\ \theta_2^2,\ \theta_1\theta_2^2 .
\]
On the line $z_1=0$,
\[
 \det R_9(0,z_2)=4096(1-z_2)^2(1-9z_2).
\]
Where $\det R_9\ne0$, every monomial of order at most 4 is congruent modulo $\mathfrak L_-$ to a holomorphic combination of the frame.
Since $\theta_1$ and $\theta_2$ map the frame into monomials of order at most 4, they act on the frame by holomorphic matrices.
This gives the logarithmic connection of Proposition~\ref{prop:minuschart}(3) near $z=0$.

For the plus chart, use $\mathcal D_2$ and $\mathcal D_3$ in place of $\mathcal L_2$ and $\mathcal L_3$, with the same multipliers, and the frame $1,\thwone,\thwtwo,\thwone^2,\thwone\thwtwo,\thwone^2\thwtwo$.
The corresponding determinant at $w=0$ equals $2^9\cdot3^{15}\ne0$.

\subsection{The plus-chart operator calculation}\label{app:pfplus}
\paragraph{Plus-chart operators.}
Write operators as sums $\sum_{\beta\in\mathbb Z_{\ge0}^2}w^\beta p_\beta(\thwone,\thwtwo)$ and multiply them with $\theta_{w_i}w^\beta=w^\beta(\theta_{w_i}+\beta_i)$.
In $\mathcal X=(3\thwone-\thwtwo+1)\mathcal D_2-9\widetilde{\mathcal L}_3$ the terms of degree zero in $w$ are $-9(3\thwone-\thwtwo+1)\thwtwo^2+9(3\thwone-\thwtwo+1)\thwtwo^2=0$, and every other term is divisible by $w_2$.
Hence $\mathcal K=w_2^{-1}\mathcal X$ is an operator, namely
\[
 \mathcal K=\thwone\thwtwo^2+w_1w_2\,\mathcal K_1+w_1w_2^2\,\mathcal K_2,
\]
with
\begin{align*}
\mathcal K_1={}&-288\thwone^3-240\thwone^2\thwtwo+256\thwone\thwtwo^2-48\thwtwo^3\\
 &-432\thwone^2-240\thwone\thwtwo+128\thwtwo^2-208\thwone-48\thwtwo-32,\\
\mathcal K_2={}&-1152\thwone^3+1152\thwone^2\thwtwo-400\thwone\thwtwo^2+48\thwtwo^3\\
 &-768\thwone^2+480\thwone\thwtwo-80\thwtwo^2-176\thwone+48\thwtwo-16.
\end{align*}
The term of degree zero of $9\mathcal K+\thwone\mathcal D_2$ is $9\thwone\thwtwo^2-9\thwone\thwtwo^2=0$, so $\mathcal D_3=w_2^{-1}(9\mathcal K+\thwone\mathcal D_2)$ is again an operator.
Expanding gives \eqref{eq:D3}.

\section{Horizontal pairing certificate}\label{app:yukawa}
The following calculation verifies the horizontal alternating form and Yukawa functions of Proposition~\ref{prop:yukawa}.
\begin{proof}[Proof of Proposition~\ref{prop:yukawa}]
We give an algebraic certificate for the assertion.
All indices below run from 0 to 5.
Put \(\Sigma_{ij}=\langle \mathbf e_i,\mathbf e_j\rangle\).
Start with
\[
 \Sigma_{05}=C_{122},\quad \Sigma_{13}=-C_{112},\quad
 \Sigma_{14}=\Sigma_{23}=-C_{122},\quad \Sigma_{24}=-C_{222},
\]
and \(\Sigma_{01}=\Sigma_{02}=\Sigma_{03}=\Sigma_{04}=\Sigma_{12}=0\).
Define the remaining upper triangular entries in the following order:
\begin{align*}
 \Sigma_{15}&=\theta_{z_1} C_{122}-(A_1)_{55}C_{122},\\
 \Sigma_{25}&=\theta_{z_2} C_{122}-(A_2)_{55}C_{122},\\
 \Sigma_{34}&=-\theta_{z_2} C_{122}-\sum_j(A_2)_{4j}\Sigma_{1j},\\
 \Sigma_{35}&=\theta_{z_2} \Sigma_{15}-\sum_j(A_2)_{5j}\Sigma_{1j},\\
 \Sigma_{45}&=\theta_{z_2} \Sigma_{25}-\sum_j(A_2)_{5j}\Sigma_{2j}.
\end{align*}
Extend by skew symmetry.
Substitution of the rational derivative reductions and \eqref{eq:Yukawas} gives the identities
\[
 \theta_{z_1} \Sigma=A_1\Sigma+\Sigma A_1^t,\qquad \theta_{z_2} \Sigma=A_2\Sigma+\Sigma A_2^t.
\]
They can be verified entirely in \(\QQ(z_1,z_2)\).
For \(\mathsf N=1-16z_1-9z_2-32z_1z_2\),
\begin{equation}
 \det \Sigma=\frac{4096\mathsf N^2\mathsf N_{122}^2}{(1-z_2)^2\Delta^4}\ne0.
 \label{eq:detpairing}
\end{equation}
Since \(\Sigma\) is horizontal, its rank on the local system is constant; it is nondegenerate at the origin, so the zeros of
\(\mathsf N\mathsf N_{122}\) are degenerations of the frame \(\mathbf e\).
At the origin the nonzero upper triangular entries are
\(\Sigma_{05}=4,\Sigma_{13}=-8,\Sigma_{14}=\Sigma_{23}=-4\), which are exactly
\(\int u^\vee v\) on \(\Rminus\), with \(u^\vee=(-1)^k u\) in degree \(k\).
This fixes the stated normalization.
Equation \eqref{eq:YukawaBoundary} is then immediate.
\end{proof}


\par\noindent{\scshape \small
Department of Mathematics \\
Faculty of Science and Engineering\\
Waseda University \\
3-4-1 Okubo, Shinjuku, Tokyo, 169-8555, Japan}
\par\noindent{\ttfamily a\_kanazawa@waseda.jp}

\end{document}